\documentclass[12pt]{article}
\usepackage{amssymb,bm}
\usepackage{amsmath}
\usepackage{amsthm}
\usepackage{graphicx}
\usepackage[active]{srcltx} 
\usepackage{hyperref}
\hypersetup{pdfborder=0 0 0}

\newtheorem{theorem}{{\sc Theorem}}[section]

\newtheorem{lemma}[theorem]{{\sc Lemma}}

\newtheorem{remark}[theorem]{Remark}

\newtheorem{definition}[theorem]{Definition}

\newcommand{\bb}[1]{\mathbb{ #1}}

\newcommand{\dOm}{\partial\Omega}

\newcommand{\bra}[1]{\overline{#1}}

\newcommand{\Trc}{\mathrm{Tr}\,}

\newcommand{\cof}{\mathrm{cof}}

\newcommand{\tns}[1]{#1\otimes #1}
\newcommand{\hf}{\displaystyle\frac{1}{2}}
\newcommand{\nth}[1]{\displaystyle\frac{1}{#1}}

\newcommand{\dif}[2]{\displaystyle\frac{\partial #1}{\partial #2}}
\newcommand{\Grad}{\nabla}
\newcommand{\Div}{\nabla \cdot}
\newcommand{\Curl}{\nabla \times}

\renewcommand{\Hat}[1]{\widehat{#1}}
\newcommand{\Tld}[1]{\widetilde{#1}}

\def\XXint#1#2#3{{\setbox0=\hbox{$#1{#2#3}{\int}$ }
\vcenter{\hbox{$#2#3$ }}\kern-.6\wd0}}

\newcommand{\lims}{\mathop{\overline\lim}}
\newcommand{\limi}{\mathop{\underline\lim}}

\newcommand{\bc}{boundary condition}

\newcommand{\rhs}{right-hand side}

\newcommand{\nbh}{neighborhood}
\newcommand{\IFF}{if and only if }

\newcommand{\Ga}{\alpha}
\newcommand{\Gb}{\beta}
\newcommand{\Gd}{\delta}
\newcommand{\Ge}{\epsilon}

\newcommand{\Gf}{\phi}

\newcommand{\Gg}{\gamma}

\newcommand{\Gl}{\lambda}

\newcommand{\Gth}{\theta}

\newcommand{\Gs}{\sigma}

\newcommand{\Gz}{\zeta}

\newcommand{\GG}{\Gamma}

\newcommand{\GO}{\Omega}

\bmdefine\BGa{\alpha}
\bmdefine\BGb{\beta}
\bmdefine\BGd{\delta}
\bmdefine\BGe{\epsilon}
\bmdefine\BGve{\varepsilon}
\bmdefine\BGf{\phi}
\bmdefine\BGvf{\varphi}
\bmdefine\BGg{\gamma}
\bmdefine\BGc{\chi}
\bmdefine\BGi{\iota}
\bmdefine\BGk{\kappa}
\bmdefine\BGl{\lambda}
\bmdefine\BGn{\eta}
\bmdefine\BGm{\mu}
\bmdefine\BGv{\nu}
\bmdefine\BGp{\pi}
\bmdefine\BGth{\theta}
\bmdefine\BGvth{\vartheta}
\bmdefine\BGr{\rho}
\bmdefine\BGvr{\varrho}
\bmdefine\BGs{\sigma}
\bmdefine\BGvs{\varsigma}
\bmdefine\BGt{\tau}
\bmdefine\BGj{\tau}
\bmdefine\BGu{\upsilon}
\bmdefine\BGo{\omega}
\bmdefine\BGx{\xi}
\bmdefine\BGy{\psi}
\bmdefine\BGz{\zeta}
\bmdefine\BGD{\Delta}
\bmdefine\BGF{\Phi}
\bmdefine\BGG{\Gamma}
\bmdefine\BGL{\Lambda}
\bmdefine\BGP{\Pi}
\bmdefine\BGT{\Theta}
\bmdefine\BGS{\Sigma}
\bmdefine\BGU{\Upsilon}
\bmdefine\BGO{\Omega}
\bmdefine\BGX{\Xi}
\bmdefine\BGY{\Psi}

\newcommand{\CA}{{\mathcal A}}
\newcommand{\CB}{{\mathcal B}}
\newcommand{\CC}{{\mathcal C}}

\newcommand{\CE}{{\mathcal E}}

\bmdefine\BCA{{\mathcal A}}
\bmdefine\BCB{{\mathcal B}}
\bmdefine\BCC{{\mathcal C}}
\bmdefine\BCD{{\mathcal D}}
\bmdefine\BCE{{\mathcal E}}
\bmdefine\BCF{{\mathcal F}}
\bmdefine\BCG{{\mathcal G}}
\bmdefine\BCH{{\mathcal H}}
\bmdefine\BCI{{\mathcal I}}
\bmdefine\BCJ{{\mathcal J}}
\bmdefine\BCK{{\mathcal K}}
\bmdefine\BCL{{\mathcal L}}
\bmdefine\BCM{{\mathcal M}}
\bmdefine\BCN{{\mathcal N}}
\bmdefine\BCO{{\mathcal O}}
\bmdefine\BCP{{\mathcal P}}
\bmdefine\BCQ{{\mathcal Q}}
\bmdefine\BCR{{\mathcal R}}
\bmdefine\BCS{{\mathcal S}}
\bmdefine\BCT{{\mathcal T}}
\bmdefine\BCU{{\mathcal U}}
\bmdefine\BCV{{\mathcal V}}
\bmdefine\BCW{{\mathcal W}}
\bmdefine\BCX{{\mathcal X}}
\bmdefine\BCY{{\mathcal Y}}
\bmdefine\BCZ{{\mathcal Z}}

\bmdefine\Bzr{ 0}
\bmdefine\Ba{ a}
\bmdefine\Bb{ b}
\bmdefine\Bc{ c}
\bmdefine\Bd{ d}
\bmdefine\Be{ e}
\bmdefine\Bf{ f}
\bmdefine\Bg{ g}
\bmdefine\Bh{ h}
\bmdefine\Bi{ i}
\bmdefine\Bj{ j}
\bmdefine\Bk{ k}
\bmdefine\Bl{ l}
\bmdefine\Bm{ m}
\bmdefine\Bn{ n}
\bmdefine\Bo{ o}
\bmdefine\Bp{ p}
\bmdefine\Bq{ q}
\bmdefine\Br{ r}
\bmdefine\Bs{ s}
\bmdefine\Bt{ t}
\bmdefine\Bu{ u}
\bmdefine\Bv{ v}
\bmdefine\Bw{ w}
\bmdefine\Bx{ x}
\bmdefine\By{ y}
\bmdefine\Bz{ z}
\bmdefine\BA{ A}
\bmdefine\BB{ B}
\bmdefine\BC{ C}
\bmdefine\BD{ D}
\bmdefine\BE{ E}
\bmdefine\BF{ F}
\bmdefine\BG{ G}
\bmdefine\BH{ H}
\bmdefine\BI{ I}
\bmdefine\BJ{ J}
\bmdefine\BK{ K}
\bmdefine\BL{ L}
\bmdefine\BM{ M}
\bmdefine\BN{ N}
\bmdefine\BO{ O}
\bmdefine\BP{ P}
\bmdefine\BQ{ Q}
\bmdefine\BR{ R}
\bmdefine\BS{ S}
\bmdefine\BT{ T}
\bmdefine\BU{ U}
\bmdefine\BV{ V}
\bmdefine\BW{ W}
\bmdefine\BX{ X}
\bmdefine\BY{ Y}
\bmdefine\BZ{ Z}

\newcommand{\SFL}{\mathsf{L}}

\title{
Scaling instability of the buckling load in axially compressed circular cylindrical shells
}
\author{Yury Grabovsky \and Davit Harutyunyan}
\begin{document}
\maketitle
\begin{abstract}
  In this paper we initiate a program of rigorous analytical investigation of
  the paradoxical buckling behavior of circular cylindrical shells under axial
  compression. This is done by the development and systematic application of
  general theory of ``near-flip'' buckling of 3D slender bodies to cylindrical
  shells. The theory predicts scaling instability of the buckling load due to
  imperfections of load. It also suggests a more dramatic scaling instability
  caused by shape imperfections. The experimentally determined scaling
  exponent 1.5 of the critical stress as a function of shell thickness appears
  in our analysis as the scaling of the lower bound on safe loads given by the
  Korn constant. While the results of this paper fall short of a definitive
  explanation of the buckling behavior of cylindrical shells, we believe that
  our approach is capable of providing reliable estimates of the buckling
  loads of axially compressed cylindrical shells.
\end{abstract}


\section{Introduction}
\label{sec:intro}
A circular cylindrical shell loaded by an axial compressive stress
will buckle producing a variety of buckling
patterns\cite{bushnell81,lcp00,dkzoa12}, including the single-dimple buckle
\cite{zmc02,hlp06}, shown in Figure~\ref{fig:diamond}.
\begin{figure}[h]
  \centering
  \includegraphics[scale=0.4]{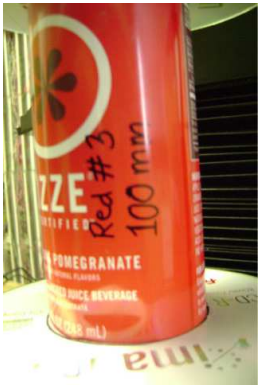}~~~~~~~~~~~~~~~~
\includegraphics[scale=0.52]{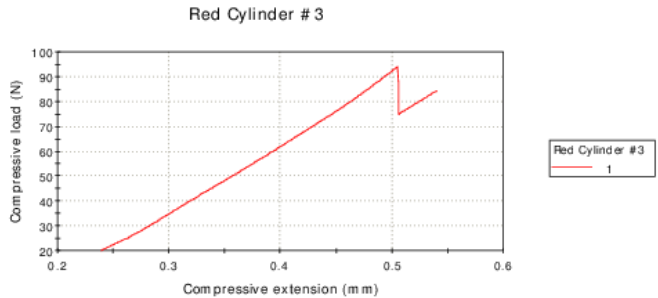}
  \caption{Single-dimple buckling pattern in buckled soda cans \cite{grig10ma}. }
  \label{fig:diamond}
\end{figure}
In the soda can experiments \cite{grig10ma} this dimple consistently appeared
with an audible click, corresponding to the drop in load in
Figure~\ref{fig:diamond} and disappears (also with a click) upon
unloading. This suggests that the local material response is still linearly
elastic, while the global non-linearity is purely geometric. The abrupt
nature of the observed buckling suggests that the trivial branch, whose stress
and strain are well-approximated by linear elasticity, becomes unstable with
respect to the observed buckling variation.

The classical shell theory supplies the following formula for the
critical stress \cite{lorenz11,timosh14} (see also \cite{tiwk59}):
\begin{equation}
  \label{clform}
  \Gs_{\rm cr}=\frac{Eh}{\sqrt{3(1-\nu^{2})}},
\end{equation}
where $E$ and $\nu$ are the Young modulus and the Poisson ratio, respectively,
and $h=t/R$ is the ratio of the wall thickness to the radius of the
cylinder. A large body of experimental results summarized in
\cite{lcp00,zmc02} show that not only the theoretical value of the buckling
load is about 4 to 5 times higher than the one observed in experiments, but
the critical stress $\Gs_{\rm cr}$ scales like $h^{3/2}$ with $h$, in stark
contradiction to (\ref{clform}). Such paradoxical behavior is generally
attributed to the sensitivity of the buckling load to imperfections of load
and shape \cite{almr63,tenn64,wms65,goev70,yama84}, due to the subcritical
nature of the bifurcation \cite{hune91,hune93,lch97,hlc99} in the
von-K\'arm\'an-Donnell equations. Yet, such an interpretation of the
experimental results does not give a quantification of sensitivity to
imperfections, and does little to explain the paradoxical $h^{1.5}$ scaling of
the critical stress. These questions have been raised in
\cite{call2000,zmc02,hlp06}, where a combination of heuristic arguments and
numerical simulations were used to address the problem.  In situations where
the classical shell theory gives predictions inconsistent with experiment, one
can question whether ``sensitivity to imperfections'' is the true source of
the inconsistency, or whether the failure of the heuristic models to
adequately describe stability of slender bodies is at play. In a companion
paper \cite{grha15} we give a rigorous proof of the asymptotical correctness
of (\ref{clform}), showing that the second variation of the energy of the
compressed shell, regarded as a 3D hyperelastic body, becomes negative when
the load exceeds the critical value (\ref{clform}).

Recent years have seen significant progress in the rigorous analysis of
dimensionally reduced theories of plates and shells based on $\GG$-convergence
\cite{fjm02,momu04,fjm06,lmap2011,lmop2011}. In this approach, one must
postulate the scaling of energy and the forces a priori, whereby different
scaling assumptions lead to different dimensionally reduced plate and shell
equations. These analyses show that the tacit assumptions of validity of
specific shell theories must be justified before conclusions about the elastic
behavior of such shells can be regarded as rigorous. By contrast, the theory
in \cite{grtr07}, based on the study of second variation, has no need for such
a priori assumptions, since it pursues a more modest goal of identifying a
critical load at which the ``trivial branch'', or ``fundamental state'', of
equilibria becomes weakly unstable. This exclusive targeting of the
instability without any attempt to compute $\GG$-limiting models or capture
a global bifurcation picture of post-buckling behavior leads to a significant
simplification in the rigorous analysis of stability of slender
structures. Most notably, our approach does not require
compactness of arbitrary low energy sequences as in
\cite{fjm02}. In particular, our method is applicable in situations where
compactness fails, as is the case for the axially compressed cylindrical
shells.

In this paper we prove that a constitutively reduced characterization of the
buckling load, derived in \cite{grtr07}, captures the buckling mode as well. A
more convenient criterion for the validity of this characterization of
buckling, derived in this paper, makes the theory applicable to a broader,
compared to \cite{grtr07}, class of slender structures, including axially
compressed cylindrical shells. While our approach is capable of providing a rigorous proof of
the classical formula (\ref{clform}) \cite{grha15}, it also reveals a possible mechanism of
imperfection sensitivity that may explain the experimental results and their
discrepancy with the classical theory. Specifically, we show that generic
imperfections in loading will change the scaling law of $\Gs_{\rm cr}$ to
$h^{5/4}$. Shape imperfections may lead to an even bigger jump in the scaling
exponent of $\Gs_{\rm cr}$ from $h$ to $h^{3/2}$. The power law $h^{3/2}$
arises as the scaling of the Korn constant \cite{grha14}, shown
to describe the universal lower bound on safe loads in \cite{grtr07}. This
explanation of the experimentally observed scaling of the critical stress
could be viewed as an improvement of the ingenuous but somewhat intuitive
arguments in \cite{call2000,zmc02}.

Generically, shape imperfections eliminate sharp bifurcation
transitions \cite{buha66}. However, the abrupt appearance of the dimple-shaped buckle
accompanied by an audible click in our experiments suggests that, in the case
of a cylindrical shell, shape imperfections \emph{do not eliminate bifurcation
  instability}. If this is indeed the case, our methods would
be able to accurately predict the critical load and the corresponding buckling
mode. However, the rigorous analysis of an imperfect cylindrical shell is beyond
the scope of this paper, since all the estimates proved here and in our
companion paper \cite{grha14} are for the specific circular cylindrical
shell geometry.

This paper is organized as follows. In Section~\ref{sec:genth} we extend the
theory of buckling of general slender bodies based on the asymptotic analysis
of second variation \cite{grtr07}. We define ``compression tensor'' and
further develop the method of buckling equivalence \cite{grtr07}.
Section~\ref{sec:ccshb} applies the theory in Section~\ref{sec:genth} and the
asymptotics of the Korn and Korn-type constants proved in \cite{grha14} to the
computation of the scaling law of the buckling load. We next demonstrate
scaling instabilities due to generic imperfections of load. We conclude the
paper with a less rigorous discussion of imperfections of shape.

\section{Buckling of slender structures}
\setcounter{equation}{0}
\label{sec:genth}
In this section we revisit the general theory of buckling developed in
\cite{grtr07} in order to extend and apply it to buckling of axially
compressed cylindrical shells. The theory provides a recipe for computing the
asymptotics of buckling loads of slender structures, as the slenderness
parameter goes to zero.

 We follow the established tradition and use the
energy criterion of stability. Namely, we say that the deformation
$\By=\By(\Bx)$, $\Bx\in\GO\subset\bb{R}^{3}$ is stable if it is a weak local
minimizer of the energy
\[
\CE(\By)=\int_{\GO}W(\Grad\By)d\Bx-\int_{\dOm}\By\cdot\Bt(\Bx)dS(\Bx),
\]
where $W(\BF)$ is the energy density function of the body and $\Bt(\Bx)$ is
the vector of dead load tractions.
The energy density $W(\BF)$ satisfies the four fundamental properties:
\begin{itemize}
\item[(P1)] Absence of prestress: $W_{\BF}(\BI)=\Bzr$;
\item[(P2)] Frame indifference: $W(\BF\BR)=W(\BF)$ for every $\BR\in SO(3)$;
\item[(P3)] Local stability of the trivial deformation $\By(\Bx)=\Bx$: $(\SFL_{0}\BGx,\BGx)\ge 0$ for any
  $\BGx\in\bb{R}^{3\times 3}$, where $\SFL_{0}=W_{\BF\BF}(\BI)$ is the linearly
  elastic tensor of material properties;
\item[(P4)] Non-degeneracy: $(\SFL_{0}\BGx,\BGx)=0$ \IFF $\BGx^{T}=-\BGx$.
\end{itemize}
Here, and elsewhere in this paper we use the notation
$(\BA,\BB)=\Trc(\BA\BB^{T})$ for the Frobenius inner product on the space of
$3\times 3$ matrices.

In \cite{grtr07} we attribute the universal nature of buckling to the two
universal properties (P1) and (P2) of the energy density function because they
guarantee non-convexity of $W(\BF)$ in any \nbh\ of the identity $\BI$. We
also remark that properties (P3) and (P4) of $\SFL_{0}$ imply a uniform lower
bound
\begin{equation}
  \label{Lcoerc}
  (\SFL_{0}\BGx,\BGx)\ge\Ga_{\SFL_{0}}|\BGx_{\rm sym}|^{2},\qquad\BGx_{\rm
    sym}=\hf(\BGx+\BGx^{T})
\end{equation}
for some $\Ga_{\SFL_{0}}>0$.

\subsection{Trivial branch}
Consider a sequence of progressively slender\footnote{The appropriate notion
  of slenderness, introduced in \cite{grtr07} is made precise in
  Defintion~\ref{def:slender}.} domains $\GO_{h}$ parametrized by a
dimensionless parameter $h$. For example, for circular cylindrical shells, $h$
is the ratio of cylinder wall thickness to the cylinder radius (we keep the
ratio of cylinder height to its radius constant). We consider a loading
program parametrized by the loading parameter $\Gl$ describing the magnitude
of the applied tractions $\Bt(\Bx;h,\Gl)=\Gl\Bt^{h}(\Bx)+O(\Gl^{2})$, as
$\Gl\to 0$, or as a measure of the prescribed strain. Here and below
$O(\Gl^{\Ga})$ is understood \emph{uniformly} in $\Bx\in\GO_{h}$ and
$h\in[0,h_{0}]$. Let $\By(\Bx;h,\Gl)$ be a family of Lipschitz equilibria of
\begin{equation}
  \label{energy}
\CE(\By;h,\Gl)=\int_{\GO_{h}}W(\Grad\By)d\Bx-\int_{\dOm_{h}}\By(\Bx)\cdot\Bt(\Bx;h,\Gl)dS(\Bx).
\end{equation}
defined on $\bra{\GO}\times[0,h]\times[0,\Gl_{0}]$ for some $h_{0}>0$ and
$\Gl_{0}>0$. The general theory can treat a wide range of \bc s. To describe
one, we restrict $\By$ to an affine subspace of $W^{1,\infty}(\GO_{h};\bb{R}^{3})$, given by
\begin{equation}
  \label{hdbc}
  \By\in\bra{\By}(\Bx;h,\Gl)+V_{h}^{\circ},
\end{equation}
where $V_{h}^{\circ}$ is a linear subspace of
$W^{1,\infty}(\GO_{h};\bb{R}^{3})$ that contains
$W_{0}^{1,\infty}(\GO_{h};\bb{R}^{3})$ and does not depend on the loading
parameter $\Gl$. The given function $\bra{\By}(\Bx;h,\Gl)\in
W^{1,\infty}(\GO_{h};\bb{R}^{3})$ describes the ``Dirichlet part'' of the \bc
s, while the traction vector $\Bt(\Bx;h,\Gl)$ describes the
Neumann-part\footnote{The use of a general subspace $V_{h}^{\circ}$ permits
  one to describe loadings in which desired linear combinations of the
  displacement and traction components are prescribed on the boundary.}.  An
example of such a description of the \bc s for the cylindrical shell will be
given in Section~\ref{sub:trbr}.
\begin{definition}
  \label{def:trbr}
  We call the family of Lipschitz equilibria $\By(\Bx;h,\Gl)$ of $\CE(\By;h,\Gl)$ a
  \textbf{linearly elastic trivial branch} if there exist $h_{0}>0$ and $\Gl_{0}>0$,
  so that for every $h\in[0,h_{0}]$ and $\Gl\in[0,\Gl_{0}]$
\begin{itemize}
\item[(i)] $\By(\Bx;h,0)=\Bx$
\item[(ii)] There exist a family of Lipschitz functions $\Bu^{h}(\Bx)$,
  independent of $\Gl$, such that
\begin{equation}
  \label{fundass}
  \|\Grad\By(\Bx;h,\Gl)-\BI-\Gl\Grad\Bu^{h}(\Bx)\|_{L^{\infty}(\GO_{h})}\le C\Gl^{2},
\end{equation}
\item[(iii)]
  \begin{equation}
    \label{reglbda}
    \|\dif{(\Grad\By)}{\Gl}(\Bx;h,\Gl)-\Grad\Bu^{h}(\Bx)\|_{L^{\infty}(\GO_{h})}\le C\Gl
  \end{equation}
\end{itemize}
where the constant $C$ is independent of $h$ and $\Gl$.
\end{definition}
We remark that neither uniqueness nor stability of the trivial branch are
assumed.

The equilibrium equations and the \bc s satisfied by the trivial branch
$\By(\Bx;h,\Gl)$ can be written explicitly in the weak form:
\begin{equation}
  \label{ELwk}
  \int_{\GO_{h}}(W_{\BF}(\Grad\By(\Bx;h,\Gl)),\Grad\BGf)d\Bx
-\int_{\dOm_{h}}\BGf\cdot\Bt(\Bx;h,\Gl)dS=0,\qquad\BGf\in V_{h}^{\circ},
\end{equation}
Differentiating (\ref{ELwk}) in $\Gl$ at $\Gl=0$, which is allowed due to
(\ref{fundass}), we obtain
\begin{equation}
  \label{linel}
  \int_{\GO_{h}}(\SFL_{0}\Grad\Bu^{h}(\Bx),\Grad\BGf)d\Bx
-\int_{\dOm_{h}}\BGf\cdot\Bt^{h}(\Bx)dS=0,\qquad\BGf\in V_{h}^{\circ},
\end{equation}
In \cite{grtr07} the notion of the near-flip buckling is defined when for any
$h\in[0,h_{0}]$ the trivial branch $\By(\Bx;h,\Gl)$ becomes unstable for
$\Gl>\Gl(h)$, where $\Gl(h)\to 0$, as $h\to 0$. This happens because it
becomes energetically more advantageous to activate bending modes rather than
store more compressive stress. This exchange of stability is detected
by the change in sign of the second variation of energy
\[
\Gd^{2}\CE(\BGf;h,\Gl)=\int_{\GO_{h}}(W_{\BF\BF}(\Grad\By(\Bx;h,\Gl))\Grad\BGf,\Grad\BGf)d\Bx,
\qquad\BGf\in V_{h},
\]
where $V_{h}=\bra{V_{h}^{\circ}}$ is the closure of $V_{h}^{\circ}$ in
$W^{1,2}(\GO_{h};\bb{R}^{3})$.
The second variation is always non negative, when $0<\Gl<\Gl(h)$ and can become
negative for some choice of the admissible variation $\BGf\in V_{h}$, when
$\Gl>\Gl(h)$. It was understood in \cite{grtr07} that this failure of weak
stability is due to properties (P1)--(P4) of $W(\BF)$ and is intimately
related to flip instability in soft device.

\subsection{Buckling load and buckling mode}
Using the second variation criterion for stability we define the buckling load
as
\begin{equation}
  \label{truebl}
  \Gl^{*}(h)=\inf\{\Gl>0:\Gd^{2}\CE(\BGf;h,\Gl)<0\text{ for some }\BGf\in
  V_{h} \}.
\end{equation}
\begin{definition}
  \label{def:nfb}
We say that the body undergoes a \textbf{near-flip buckling} if $\Gl^{*}(h)>0$
for all $h\in(0,h_{0})$, for some $h_{0}>0$, and $\Gl^{*}(h)\to 0$, as $h\to 0$.
\end{definition}
We refer to \cite{grtr07} for a discussion of why this terminology is appropriate.

The buckling mode is generally understood as the variation $\BGf^{*}_{h}\in
V_{h}\setminus\{0\}$, such that
$\Gd^{2}\CE(\BGf^{*}_{h};h,\Gl^{*}(h))=0$\footnote{The question of existence
  of the buckling mode $\BGf^{*}_{h}$ is irrelevant here, since the goal of this discussion
  is to explain the intuitive meaning of the formal definition of a buckling
  mode, made in Definition~\ref{def:asymbm}.}. However, if we are only interested
in the \emph{asymptotics} of the critical load, as $h\to 0$, then we would not
distinguish between $\Gl^{*}(h)$ and $\Gl(h)$, as long as
$\Gl(h)/\Gl^{*}(h)\to 1$, as $h\to 0$. If we replace $\Gl^{*}(h)$ with
$\Gl_{\Ge}(h)=\Gl^{*}(h)(1+\Ge)$, then we estimate
\[
\Gd^{2}\CE(\BGf^{*}_{h};h,\Gl^{*}(h)(1+\Ge))\approx\Gl^{*}(h)\Ge
\dif{(\Gd^{2}\CE)}{\Gl}(\BGf^{*}_{h};h,\Gl^{*}(h)).
\]
This means that for the purposes of asymptotics we should not distinguish
differences in values of second variation that are infinitesimal, compared to
\[
\Gl^{*}(h)\dif{(\Gd^{2}\CE)}{\Gl}(\BGf^{*}_{h};h,\Gl^{*}(h)).
\]
In keeping with these observations, we redefine the notion of the buckling
load and buckling mode, under the assumption that the body undergoes a
near-flip buckling in the sense of Definition~\ref{def:nfb}.
\begin{definition}
  \label{def:asymbm}
We say that $\Gl(h)\to 0$, as $h\to 0$ is a buckling load if
\begin{equation}
  \label{asymbl}
  \lim_{h\to 0}\frac{\Gl(h)}{\Gl^{*}(h)}=1.
\end{equation}
A \textbf{buckling mode} is a family of variations $\BGf_{h}\in V_{h}\setminus\{0\}$, such that
\begin{equation}
  \label{trubm}
\lim_{h\to 0}\frac{\Gd^{2}\CE(\BGf_{h};h,\Gl^{*}(h))}
{\Gl^{*}(h)\dif{(\Gd^{2}\CE)}{\Gl}(\BGf_{h};h,\Gl^{*}(h))}=0.
\end{equation}
\end{definition}

The most important insight in \cite{grtr07} is that at the critical load
$\Gl^{*}(h)\to 0$, as $h\to 0$, the local material response is well inside the
linearly elastic regime and the instability can be detected by a simpler
\emph{constitutively linearized} second variation:
\begin{equation}
  \label{ccsv}
\Gd^{2}\CE_{cl}(\BGf;h,\Gl)=\int_{\GO_{h}}\{(\SFL_{0}e(\BGf),e(\BGf))+
\Gl(\BGs_{h},\Grad\BGf^{T}\Grad\BGf)\}d\Bx, \qquad\BGf\in V_{h}
\end{equation}
where $e(\BGf)=\hf(\Grad\BGf+(\Grad\BGf)^{T})$ and
\begin{equation}
  \label{minusstress}
 \BGs_{h}(\Bx)=\SFL_{0}e(\Bu^{h}(\Bx))
\end{equation}
is the linear elastic stress. We define
\begin{equation}
\label{compres.measure}
\mathfrak{C}_{h}(\BGf)=\dif{(\Gd^{2}\CE_{cl})}{\Gl}(\BGf;h,\Gl)=
\int_{\GO_{h}}(\BGs_{h},\Grad\BGf^{T}\Grad\BGf)d\Bx.
\end{equation}
Observe that
\begin{equation}
  \label{Adef}
\CA_{h}=\left\{\BGf\in V_{h}:\mathfrak{C}_{h}(\BGf)<0\right\}
\end{equation}
can be regarded as the set of all destabilizing variations for
(\ref{ccsv}). We assume that the applied loading has a compressive nature. In
particular, we assume that the sets $\CA_{h}$ are non-empty for all
$h\in(0,h_{0})$ for some $h_{0}>0$. In parallel with our discussion of the
asymptotics of the critical load and buckling mode we define the functional
\begin{equation}
  \label{Kgen}
\mathfrak{R}(h,\BGf)=-\frac{\int_{\GO_{h}}(\SFL_{0}e(\BGf),e(\BGf))d\Bx}
{\int_{\GO_{h}}(\BGs_{h},\Grad\BGf^{T}\Grad\BGf)d\Bx}=
-\frac{\mathfrak{S}_{h}(\BGf)}{\mathfrak{C}_{h}(\BGf)},
\end{equation}
where
\begin{equation}
\label{stabil.measure}
\mathfrak{S}_{h}(\BGf)=\int_{\GO_{h}}(\SFL_{0}e(\BGf),e(\BGf))d\Bx.
\end{equation}
is the measure of stability of the trivial branch.
The constitutively linearized buckling load and buckling mode are then defined
by analogy with the original second variation.
\begin{definition}
  \label{def:Bload}
The \textbf{constitutively linearized buckling load} $\Hat{\Gl}(h)$ is defined by
\begin{equation}
  \label{clin}
\Hat{\Gl}(h)=\inf_{\BGf\in\CA_{h}}\mathfrak{R}(h,\BGf).
\end{equation}
  We say that the family of variations $\{\BGf_{h}\in\CA_{h}:h\in(0,h_{0})\}$
  is a \textbf{constitutively linearized buckling mode} if
  \begin{equation}
    \label{clbm}
\lim_{h\to 0}\frac{\mathfrak{R}(h,\BGf_{h})}{\Hat{\Gl}(h)}=1.
  \end{equation}
\end{definition}
We now need to prove that under some reasonable assumptions the constitutively
linearized buckling load and buckling mode are buckling mode and buckling
mode, respectively, in the sense of Definition~\ref{def:asymbm}.

Recall the definition of the Korn constant
\begin{equation}
  \label{KK}
  K(V_{h})=\inf_{\BGf\in V_{h}}\frac{\|e(\BGf)\|^{2}}{\|\Grad\BGf\|^{2}}.
\end{equation}
Here and elsewhere in this
paper $\|\cdot\|$ always denotes the $L^{2}$-norm on $\GO_{h}$.
\begin{definition}
  \label{def:slender}
We say that the  body $\GO_{h}$ is \textbf{slender} if
\begin{equation}
  \label{slender}
  \lim_{h\to 0}K(V_{h})=0.
\end{equation}
\end{definition}
We remark that this notion of slenderness, introduced in \cite{grtr07}, is not
purely geometric, but depends on the type of loading described by the subspace
$V_{h}$. On the one hand, a thin rod or a plate in the hard device will not be
regarded as slender, since their Korn constant is 1/2, regardless of their
geometric slenderness. On the other hand, a geometrically non-slender body,
such as a ball or a cube will not be
slender under our definition, for any \bc s that exclude all rigid body motions.

\begin{theorem}[Asymptotics of the critical load]
  \label{th:crit}
Suppose that the body is slender in the sense of
Definition~\ref{def:slender}. Assume that the constitutively linearized
critical load $\Hat{\Gl}(h)$, defined in (\ref{clin}) satisfies
$\Hat{\Gl}(h)>0$ for all sufficiently small $h$ and
\begin{equation}
  \label{sufcond}
  \lim_{h\to 0}\frac{\Hat{\Gl}(h)^{2}}{K(V_{h})}=0.
\end{equation}
Then $\Hat{\Gl}(h)$ is the buckling load and any constitutively linearized
buckling mode $\BGf_{h}$ is a buckling mode in the sense of Definition~\ref{def:asymbm}.
\end{theorem}
  The theorem is proved by means of the basic estimate, which is a simple
  modification of the estimates in \cite{grtr07} used in the derivation of the
  formula for $\Gd^{2}\CE_{cl}(\BGf;h,\Gl)$:
\begin{lemma}
  \label{leb:be}
Suppose $\By(\Bx;h,\Gl)$ satisfies (\ref{fundass}) and $W(\BF)$ has the
properties (P1)--(P4). Then
\begin{equation}
  \label{basic}
\left|\Gd^{2}\CE(\BGf;h,\Gl)-\Gd^{2}\CE_{cl}(\BGf;h,\Gl)\right|\le
C\left(\frac{\Gl}{\sqrt{K(V_{h})}}+\frac{\Gl^{2}}{K(V_{h})}\right)\mathfrak{S}_{h}(\BGf).
\end{equation}
and
\begin{equation}
  \label{dbasic}
\left|\dif{(\Gd^{2}\CE)}{\Gl}(\BGf;h,\Gl)-\mathfrak{C}_{h}(\BGf)\right|\le
C\left(\frac{1}{\sqrt{K(V_{h})}}+\frac{\Gl}{K(V_{h})}\right)\mathfrak{S}_{h}(\BGf).
\end{equation}
where the constant $C$ is independent of $h$, $\Gl$ and $\BGf$.
\end{lemma}
\begin{proof}
According to the frame indifference property (P2),
$W(\BF)=\Hat{W}(\BF^{T}\BF)$. Differentiating this formula twice we obtain
\begin{equation}
  \label{svobj}
(W_{\BF\BF}(\BF)\BGx,\BGx)=4(\Hat{W}_{\BC\BC}(\BC)(\BF^{T}\BGx),\BF^{T}\BGx)+
2(\Hat{W}_{\BC}(\BC),\BGx^{T}\BGx),\qquad\BC=\BF^{T}\BF.
\end{equation}
We can estimate
\[
|(\Hat{W}_{\BC\BC}(\BC)(\BF^{T}\BGx),\BF^{T}\BGx)-(\Hat{W}_{\BC\BC}(\BI)\BGx,\BGx)|\le
|(\Hat{W}_{\BC\BC}(\BC)(\BF^{T}-\BI)\BGx,(\BF^{T}-\BI)\BGx)|+
\]
\[
|((\Hat{W}_{\BC\BC}(\BC)-\Hat{W}_{\BC\BC}(\BI))\BGx,\BGx)|+
2|(\Hat{W}_{\BC\BC}(\BC)\BGx,(\BF^{T}-\BI)\BGx)|
\]
When $\BF$ is uniformly bounded we obtain
\[
|(\Hat{W}_{\BC\BC}(\BC)(\BF^{T}\BGx),\BF^{T}\BGx)-(\Hat{W}_{\BC\BC}(\BI)\BGx,\BGx)|\le
C\left(|\BF-\BI|^{2}|\BGx|^{2}+|\BC-\BI||\BGx_{\rm sym}|^{2}+|\BF-\BI||\BGx_{\rm sym}||\BGx|
\right).
\]
Similarly,
\[
|(\Hat{W}_{\BC}(\BC)-\Hat{W}_{\BC\BC}(\BI)(\BC-\BI),\BGx^{T}\BGx)|\le
C|\BC-\BI|^{2}|\BGx|^{2}
\]
When $\BF=\Grad\By(\Bx;h,\Gl)$ and $\BGx=\Grad\BGf$ we obtain, taking into
account (\ref{fundass}), that
\[
|\BF-\BI|\le C\Gl,\qquad|\BC-\BI|\le C\Gl.
\]
Observing that
\[
4\Hat{W}_{\BC\BC}(\BI)=W_{\BF\BF}(\BI)=\SFL_{0},\qquad|\BC-\BI-2\Gl e(\Bu^{h})|\le C\Gl^{2}.
\]
we obtain the estimate
\[
|(W_{\BF\BF}(\BF)\BGx,\BGx)-(\SFL_{0}\BGx_{\rm sym},\BGx_{\rm sym})-\Gl(\BGs_{h},\BGx^{T}\BGx)|\le
C(\Gl|\BGx_{\rm sym}||\BGx|+\Gl^{2}|\BGx|^{2}).
\]
Integrating over $\GO_{h}$ as using the coercivity (\ref{Lcoerc}) of
$\SFL_{0}$ we obtain the estimate (\ref{basic}).

In order to prove the estimate (\ref{dbasic}) we substitute
$\BF=\Grad\By(\Bx;h,\Gl)$ and $\BGx=\Grad\BGf$ into (\ref{svobj}) and
differentiate in $\Gl$, obtaining
\[
\dif{(W_{\BF\BF}(\BF)\BGx,\BGx)}{\Gl}=
4((\Hat{W}_{\BC\BC\BC}(\BC)\dot{\BC})(\BF^{T}\BGx),\BF^{T}\BGx)+
8(\Hat{W}_{\BC\BC}(\BC)(\BF^{T}\BGx),\dot{\BF}^{T}\BGx)+
2(\Hat{W}_{\BC\BC}(\BC)\dot{\BC},\BGx^{T}\BGx),
\]
where $\dot{\BC}$ and $\dot{\BF}$ denote differentiation with respect to $\Gl$.
Using the uniform boundedness of $\dot{\BC}$, which is a corollary of
(\ref{reglbda}), as well as (\ref{fundass}) we estimate
\[
|((\Hat{W}_{\BC\BC\BC}(\BC)\dot{\BC})(\BF^{T}\BGx),\BF^{T}\BGx)|\le
C(|\BGx_{\rm sym}|^{2}+\Gl|\BGx||\BGx_{\rm sym}|).
\]
and
\[
|(\Hat{W}_{\BC\BC}(\BC)(\BF^{T}\BGx),\dot{\BF}^{T}\BGx)|\le
C(|\BGx||\BGx_{\rm sym}|+\Gl|\BGx|^{2}).
\]
We also estimate, using $|\BC-\BI|\le C\Gl$ and $|\dot{\BC}-2e(\Bu^{h})|\le
C\Gl$, that are consequences of (\ref{fundass}) and (\ref{reglbda}):
\[
|2(\Hat{W}_{\BC\BC}(\BC)\dot{\BC},\BGx^{T}\BGx)-(\BGs_{h},\BGx^{T}\BGx)|\le C\Gl|\BGx|^{2}.
\]
\end{proof}

\begin{proof}[Proof of Theorem~\ref{th:crit}]
By definition of $\Hat{\Gl}(h)$,  for any
 $\Ge>0$ and any $h\in (0,h_{0})$ there exists
$\BGf_h\in \CA_h$ such that
\begin{equation}
\label{1+epsilon}
\mathfrak{S}_{h}(\BGf_{h})+\Hat{\Gl}(h)(1+\epsilon)\mathfrak{C}_{h}(\BGf_{h})<0,
\end{equation}
thus,
\[
\Gd^{2}\CE_{cl}(\BGf_{h};h,\Hat{\Gl}(h)(1+2\Ge))\le-\frac{\Ge\mathfrak{S}_{h}(\BGf_{h})}{1+\Ge}.
\]
The estimate
(\ref{basic}) gives the upper bound on the second variation:
\[
\Gd^{2}\CE(\BGf_h;h,\Hat{\Gl}(h)(1+2\Ge))\leq\left(-\frac{\epsilon}{(1+\epsilon)}+
C\left(\frac{\Hat{\Gl}(h)}{\sqrt{K(V_{h})}}+\frac{\Hat{\Gl}(h)^{2}}{K(V_{h})}\right)
\right)\mathfrak{S}_{h}(\BGf_{h}),
\]
Thus, due to (\ref{sufcond}), for sufficiently small $h$, we have
$\Gd^{2}\CE(\BGf_h;h,\Hat{\Gl}(h)(1+2\Ge))<0$, and hence
$\Gl^{*}(h)\le\Hat{\Gl}(h)(1+2\Ge)$. We conclude that
\begin{equation}
  \label{limsup}
\lims_{h\to 0}\frac{\Gl^{*}(h)}{\Hat{\Gl}(h)}\le 1.
\end{equation}

To prove the opposite inequality we observe that
by definition of $\Hat{\Gl}(h)$ we have
\[
\mathfrak{S}_{h}(\BGf)+\Hat{\Gl}(h)\mathfrak{C}_{h}(\BGf)\ge 0
\]
for any $\BGf\in V_{h}$. Therefore, for any $\Ge>0$ and any
$0<\Gl\le\Hat{\Gl}(h)(1-\Ge))$ we have
\[
\Gd^{2}\CE_{cl}(\BGf;h,\Gl)\ge\Ge\mathfrak{S}_{h}(\BGf).
\]
The estimate (\ref{basic}) now gives the lower bound on the second variation:
\[
\Gd^{2}\CE(\BGf;h,\Gl)\ge\left(\epsilon-
C\left(\frac{\Hat{\Gl}(h)}{\sqrt{K(V_{h})}}+\frac{\Hat{\Gl}(h)^{2}}{K(V_{h})}\right)
\right)\mathfrak{S}_{h}(\BGf),
\]
Thus for all sufficiently small $h$ and all $\BGf\in V_{h}\setminus\{0\}$ we
have $\Gd^{2}\CE(\BGf;h,\Gl)>0$ for all
$0<\lambda\leq\Hat{\Gl}(h)(1-\epsilon)$, which means that
$\Gl^{*}(h)\ge\Hat{\Gl}(h)(1-\epsilon)$. This implies
\begin{equation}
  \label{liminf}
\limi_{h\to 0}\frac{\Gl^{*}(h)}{\Hat{\Gl}(h)}\ge 1.
\end{equation}
Combining (\ref{limsup}) and (\ref{liminf}) we conclude that $\Hat{\Gl}(h)$ is
the buckling load.

Assume now that $\BGf_h$ is a constitutively linearized buckling mode,
i.e. (\ref{clbm}) holds. Set $\Gl=\Gl^{*}(h)$ and $\BGf=\BGf_h$ in the inequality
(\ref{basic}).
Then, dividing both sides of the inequality by
$-\Gl^{*}(h)\mathfrak{C}_{h}(\BGf_{h})>0$ we obtain
\[
\left|\frac{\Gd^{2}\CE(\BGf_{h};h,\Gl^{*}(h))}{-\Gl^{*}(h)\mathfrak{C}_{h}(\BGf_{h})}-
\left(\frac{\mathfrak{R}(h,\BGf_{h})}{\Gl^{*}(h)}-1\right)\right|\le
C\left(\frac{\Gl^{*}(h)}{\sqrt{K(V_{h})}}+\frac{(\Gl^{*}(h))^{2}}{K(V_{h})}\right)
\frac{\mathfrak{R}(h,\BGf_{h})}{\Gl^{*}(h)}.
\]
Since we have proved that $\Hat{\Gl}(h)$ is the buckling load we conclude that
\[
\lim_{h\to 0}\frac{\Gd^{2}\CE(\BGf_{h};h,\Gl^{*}(h))}{\Gl^{*}(h)\mathfrak{C}_{h}(\BGf_{h})}=0.
\]
Similarly, setting $\Gl=\Gl^{*}(h)$ and $\BGf=\BGf_h$ in the inequality
(\ref{dbasic}) and dividing both sides of the inequality by
$-\mathfrak{C}_{h}(\BGf_{h})>0$ we obtain
\[
\left|\frac{\dif{(\Gd^{2}\CE)}{\Gl}(\BGf_{h};h,\Gl^{*}(h))}{-\mathfrak{C}_{h}(\BGf_{h})}
+1\right|\le
C\left(\frac{\Gl^{*}(h)}{\sqrt{K(V_{h})}}+\frac{(\Gl^{*}(h))^{2}}{K(V_{h})}\right)
\frac{\mathfrak{R}(h,\BGf_{h})}{\Gl^{*}(h)}.
\]
We conclude that
\[
\lim_{h\to 0}\frac{\dif{(\Gd^{2}\CE)}{\Gl}(\BGf_{h};h,\Gl^{*}(h))}{\mathfrak{C}_{h}(\BGf_{h})}=1.
\]
It follows now that $\BGf_{h}$ satisfies (\ref{trubm}), and the theorem is proved.
\end{proof}
An immediate consequence of our Rayleigh quotient characterization of buckling
load (\ref{clin}) is the \emph{safe load} estimate:
\[
\Hat{\Gl}(h)=\inf_{\BGf\in\CA_{h}}\mathfrak{R}(h,\BGf)\ge\inf_{\BGf\in V_{h}}
\frac{\mathfrak{S}_{h}(\BGf)}{\|\BGs_{h}\|_{\infty}\|\nabla\BGf\|^{2}}=
\frac{K_{\SFL_{0}}(V_{h})}{\|\BGs_{h}\|_{\infty}},
\]
where
\[
K_{\SFL_{0}}(V_{h})=\inf_{\BGf\in V_{h}}
\frac{\int_{\GO_{h}}(\SFL_{0}e(\BGf),e(\BGf))d\Bx}{\|\Grad\BGf\|^{2}}.
\]
We remark that rods and plates have buckling loads proportional to
$K_{\SFL_{0}}(V_{h})$, while the theoretical buckling load for axially
compressed circular cylindrical shells is much higher. Based on this, we
conjecture that buckling loads that scale with $K(V_{h})$ should not exhibit
sensitivity to imperfections.

\subsection{Buckling equivalence}
\label{sub:be}
In the previous subsection we showed that the asymptotics of the critical load
and buckling mode can be captured by a constitutively linearized functional
$\mathfrak{R}(h,\BGf)$. Even though such a characterization of buckling
represents a significant simplification, compared to the characterization
based on the second variation of a fully non-linear energy functional, further
simplifications may be necessary to obtain explicit analytic expressions in
specific problems.  We envision two ways in
which the analysis of buckling can be simplified. One is the simplification of
the functional $\mathfrak{R}(h,\BGf)$. The other is replacing the space of all
admissible functions $\CA_{h}$ with a smaller space $\CB_{h}$. For example, we
may want to use a specific ansatz, like the Kirchhoff ansatz in buckling of
rods and plates. In order to formalize our simplification procedure we make
the following definitions.
\begin{definition}
  \label{def:equiv}
  Assume that $J(h,\BGf)$ is a variational functional defined on
  $\CB_h\subset\CA_{h}$. We say that the pair $(\CB_h, J(h,\BGf))$
  \textbf{characterizes buckling} if the following three conditions are
  satisfied
\begin{enumerate}
\item[(a)] Characterization of the buckling load:
If
\[
\Gl(h)=\inf_{\BGf\in\CB_{h}}J(h,\BGf),
\]
then $\Gl(h)$ is a buckling load in the sense of Definition~\ref{def:asymbm}.
\item[(b)] Characterization of the buckling mode:
If $\BGf_{h}\in\CB_{h}$ is a buckling mode in the sense of Definition~\ref{def:asymbm}, then
  \begin{equation}
    \label{bmode}
\lim_{h\to 0}\frac{J(h,\BGf_{h})}{\Gl(h)}=1.
  \end{equation}
\item[(c)] Faithful representation of the buckling mode:
If $\BGf_{h}\in\CB_{h}$ satisfies (\ref{bmode}) then it is a buckling mode.
\end{enumerate}
\end{definition}
\begin{definition}
  \label{def:Bequivalence}
  Two pairs $(\CB_h, J_{1}(h,\BGf))$ and $(\CC_h, J_{2}(h,\BGf))$ are called
  \textbf{buckling equivalent} if the pair $(\CB_h,J_{1}(h,\BGf))$ characterizes buckling if
  and only if $(\CC_h,J_2(h,\BGf))$ does.
\end{definition}
The notion of buckling equivalence of \emph{functionals} $(\CB_h, J(h,\BGf))$
is an extension of B-equivalence, introduced in \cite{grtr07}, in that it also
captures buckling modes in addition to buckling loads.

Let us first address a simple question of restricting the space of functions
$\CB_{h}$ to an ``ansatz'' $\CC_{h}$.
\begin{lemma}
\label{lem:pairB_hJ}
Suppose the pair $(\CB_{h},J(h,\BGf))$ characterizes buckling.
Let $\CC_{h}\subset\CB_{h}$ be such that it contains a buckling mode. Then the
pair $(\CC_{h},J(h,\BGf))$ characterizes buckling.
\end{lemma}
\begin{proof}
Let
\[
\Gl(h)=\inf_{\BGf\in\CB_{h}}J(h,\BGf),\qquad\Tld{\Gl}(h)=\inf_{\BGf\in\CC_{h}}J(h,\BGf).
\]
Then, clearly, $\Tld{\Gl}(h)\ge\Gl(h)$. By assumption there exists a
buckling mode $\BGf_{h}\in\CC_{h}\subset\CB_{h}$. Therefore,
\[
\lims_{h\to 0}\frac{\Tld{\Gl}(h)}{\Gl(h)}\le
\lim_{h\to 0}\frac{J(h,\BGf_{h})}{\Gl(h)}=1,
\]
since the pair $(\CB_{h},J(h,\BGf))$ characterizes buckling. Hence
\begin{equation}
  \label{adeq}
 \lim_{h\to 0}\frac{\Tld{\Gl}(h)}{\Gl(h)}=1,
\end{equation}
and part (a) of Definition~\ref{def:equiv} is established.

If $\BGf_{h}\in\CC_{h}\subset\CB_{h}$ is a buckling mode then
\[
\lim_{h\to 0}\frac{J(h,\BGf_{h})}{\Gl(h)}=1,
\]
since the pair $(\CB_{h},J(h,\BGf))$ characterizes buckling. Part (b) now
follows from (\ref{adeq}).

Finally, if $\BGf_{h}\in\CC_{h}$ satisfies
\[
\lim_{h\to 0}\frac{J(h,\BGf_{h})}{\Tld{\Gl}(h)}=1,
\]
then, $\BGf_{h}\in\CB_{h}$ and by (\ref{adeq}) we also have
\[
\lim_{h\to 0}\frac{J(h,\BGf_{h})}{\Gl(h)}=1.
\]
Therefore, $\BGf_{h}$ is a buckling mode. The Lemma is proved now.
\end{proof}
Our key tool for simplification of the functionals $J(h,\BGf)$ characterizing
buckling is the following theorem.
\begin{theorem}[Buckling equivalence]
\label{th:Bequivalence}
Suppose that $\Gl(h)$ is a buckling load in the sense of Definition~\ref{def:asymbm}.
If either
\begin{equation}
\label{J1J2}
\lim_{h\to 0}\Gl(h)\sup_{\BGf\in\CB_{h}}\left|\frac{1}{J_1(h,\BGf)}-\frac{1}{J_2(h,\BGf)}\right|=0,
\end{equation}
or
\begin{equation}
\label{J2J1}
\lim_{h\to 0}\nth{\Gl(h)}\sup_{\BGf\in\CB_{h}}|J_1(h,\BGf)-J_2(h,\BGf)|=0,
\end{equation}
then the pairs $(\CB_h, J_1(h,\BGf))$ and $(\CB_h, J_2(h,\BGf))$ are
buckling equivalent in the sense of Definition~\ref{def:Bequivalence}.
\end{theorem}
\begin{proof}
Let us introduce the following notation:
\[
\Gl_i(h)=\inf_{\BGf\in\CB_h}J_i(h,\BGf),\quad i=1,2.
\]
\[
\Gd_{1}(h)=\Gl(h)\sup_{\BGf\in\CB_{h}}\left|\frac{1}{J_1(h,\BGf)}-\frac{1}{J_2(h,\BGf)}\right|.
\]
\[
\Gd_{2}(h)=\nth{\Gl(h)}\sup_{\BGf\in\CB_{h}}|J_1(h,\BGf)-J_2(h,\BGf)|.
\]
Then
\[
\left|\frac{\Gl(h)}{\Gl_1(h)}-\frac{\Gl(h)}{\Gl_2(h)}\right|=
\Gl(h)\left|\sup_{\BGf\in\CB_{h}}\frac{1}{J_1(h,\BGf)}
-\sup_{\BGf\in\CB_{h}}\frac{1}{J_2(h,\BGf)}\right|\le\Gd_{1}(h)
\]
and
\[
\frac{|\Gl_1(h)-\Gl_{2}(h)|}{\Gl(h)}=
\nth{\Gl(h)}\left|\inf_{\BGf\in\CB_{h}}J_1(h,\BGf)
-\inf_{\BGf\in\CB_{h}}J_2(h,\BGf)\right|\le\Gd_{2}(h)
\]
Assume that $(\CB_h, J_1(h,\BGf))$ characterizes buckling. Then we have just proved
that if either $\Gd_{1}(h)\to 0$ or $\Gd_{2}(h)\to 0$, as $h\to 0$, then
$\Gl_2(h)/\Gl(h)\to 1$, as $h\to 0$, and condition (a) in
Definition~\ref{def:equiv} is proved for $J_{2}(h,\BGf)$.

Observe that by parts (b) and (c) of Definition~\ref{def:equiv}
$\BGf_{h}\in\CB_{h}$ is the buckling mode \IFF
\[
\lim_{h\to 0}\frac{J_{1}(h,\BGf_{h})}{\Gl_{1}(h)}=1.
\]
This is equivalent to
\[
\lim_{h\to 0}\frac{\Gl(h)}{J_{1}(h,\BGf_{h})}=1.
\]
Therefore,
\[
\lim_{h\to 0}\frac{J_{2}(h,\BGf_{h})}{\Gl(h)}=1,
\]
since either
\[
\left|\frac{\Gl(h)}{J_{1}(h,\BGf_{h})}-\frac{\Gl(h)}{J_{2}(h,\BGf_{h})}\right|\le\Gd_{1}(h)
\]
or
\[
\frac{|J_{1}(h,\BGf_{h})-J_{2}(h,\BGf_{h})|}{\Gl(h)}\le\Gd_{2}(h)
\]
Thus, in view of part (a), $\BGf_{h}$ is a buckling mode \IFF
\[
\lim_{h\to 0}\frac{J_{2}(h,\BGf_{h})}{\Gl_{2}(h)}=1.
\]
\end{proof}

As an application of Theorem~\ref{th:Bequivalence} we show that we can
simplify the Rayleigh quotient $\mathfrak{R}(h,\BGf)$ further.
\begin{theorem}
  \label{th:rqs}
Suppose that the critical load $\Hat{\Gl}(h)$ satisfies (\ref{sufcond}). Let
\[
\mathfrak{R}_{0}(h,\BGf)=-\frac{\int_{\GO_{h}}(\SFL_{0}e(\BGf),e(\BGf))d\Bx}
{\nth{4}\int_{\GO_{h}}(\Tld{\BGs}_{h}\Curl\BGf,\Curl\BGf)d\Bx}=
-\frac{\mathfrak{S}_{h}(\BGf)}{\mathfrak{C}_{h}^{0}(\BGf)},
\]
where
\begin{equation}
  \label{compten}
\Tld{\BGs}_{h}=(\Trc\BGs_{h})\BI-\BGs_{h}
\end{equation}
is the compression tensor.
Then
$(\CA_{h},\mathfrak{R}(h,\BGf))$ and $(\CA_{h},\mathfrak{R}_{0}(h,\BGf))$ are
buckling equivalent.
\end{theorem}
\begin{proof}
  For $\Ba\in\bb{R}^{3}$, let $\pi(\Ba)$ denote a $3\times 3$ antisymmetric
  matrix defined by the cross-product map:
\[
\pi(\Ba)\Bu=\Ba\times\Bu.
\]
Then $\Grad\BGf-(\Grad\BGf)^{T}=\pi(\Curl\BGf)$.
We observe that replacing $\Grad\BGf$ with $e(\BGf)-\pi(\Curl\BGf)/2$ in
$\mathfrak{C}_{h}(\BGf)$ we obtain
\[
\mathfrak{C}_{h}(\BGf)=\int_{\GO_{h}}\left(\BGs_{h},e(\BGf)^{2}+e(\BGf)\pi(\Curl\BGf)\right)d\Bx
+\mathfrak{C}^{0}_{h}(\BGf).
\]
It follows that for every $\BGf\in V_{h}$
\[
|\mathfrak{C}_{h}(\BGf)-\mathfrak{C}^{0}_{h}(\BGf)|\le\|\BGs_{h}\|_{\infty}(\|e(\BGf)\|^{2}
+2\|e(\BGf)\|\|\Grad\BGf\|)\le\|\BGs_{h}\|_{\infty}\|e(\BGf)\|^{2}\left(1+\frac{2}{\sqrt{K(V_{h})}}\right).
\]
Recalling that, due to (\ref{Lcoerc}),
\[
\mathfrak{S}_{h}(\BGf)\ge\Ga_{\SFL_{0}}\|e(\BGf)\|^{2}
\]
we obtain
\[
\Hat{\Gl}(h)\left|\nth{\mathfrak{R}(h,\BGf)}-\nth{\mathfrak{R}_{0}(h,\BGf))}\right|\le
\frac{\|\BGs_{h}\|_{\infty}}{\Ga_{\SFL_{0}}}
\left(\Hat{\Gl}(h)+\frac{2\Hat{\Gl}(h)}{\sqrt{K(V_{h})}}\right).
\]
Thus (\ref{sufcond}) implies that the sufficient condition (\ref{J1J2}) for
buckling equivalence is satisfied. The theorem is proved.
\end{proof}
We remark that $\Curl\BGf$ is a scalar in 2D, and similar calculations show
that the functional $\mathfrak{R}(h,\BGf)$ can be replaced in 2D by
\[
\mathfrak{R}_{0}^{2D}(h,\BGf)=-\frac{\int_{\GO_{h}}(\SFL_{0}e(\BGf),e(\BGf))d\Bx}
{\nth{2}\int_{\GO_{h}}\Trc\BGs_{h}|\Grad\BGf|^{2}d\Bx}
\]
Therefore, in the case of a homogeneous compressive trivial branch
\[
\lim_{h\to 0}\Trc\BGs_{h}=\mathfrak{c}<0
\]
we have a general formula for the critical load \cite{grtr07}:
\[
\Hat{\Gl}(h)=\frac{2K_{\SFL_{0}}(V_{h})}{\mathfrak{c}}.
\]
By contrast, the situation in 3D is much more nuanced. Even in the case of a
homogeneous trivial branch, the critical load formula demands further study.

\section{Buckling of circular cylindrical shells}
\setcounter{equation}{0}
\label{sec:ccshb}
In this section we apply the theory of near-flip buckling developed in
Section~\ref{sec:genth} to the buckling of circular cylindrical shells under
axial compression.

\subsection{Trivial branch }
\label{sub:trbr}
Consider the circular cylindrical shell given in cylindrical coordinates
$(r,\Gth,z)$ as follows:
\[
\CC_{h}=I_{h}\times\bb{T}\times[0,L],\qquad I_{h}=[1-h/2,1+h/2],
\]
where $\bb{T}$ is a 1-dimensional torus (circle) describing $2\pi$-periodicity in
$\Gth$. In this paper we consider the axial compression of the shell where
the deformation $\By:\CC_{h}\to \mathbb R^3$ satisfies the following boundary
conditions:
\begin{equation}
  \label{bc}
  y_{\Gth}(r,\Gth,0)=y_{z}(r,\Gth,0)=y_{\Gth}(r,\Gth,L)=0,\quad
  y_{z}(r,\Gth,L)=(1-\Gl)L,\quad\Bt(\Bx;h,\Gl)=\Bzr,
\end{equation}
where $\Bt$ is the vector of tractions in
(\ref{energy}). The loading is parametrized by the compressive strain
$\Gl$ in the axial direction. To apply our theory of buckling we need to
describe the \bc s in the form (\ref{hdbc}). This is done by defining
\[
\bra{\By}(\Bx;h,\Gl)=(1-\Gl)z\Be_{z},
\]
and
\[
V_{h}^{\circ}=\{\BGf\in W^{1,\infty}(\CC_{h};\mathbb R^3):
\phi_{\Gth}(r,\Gth,0)=\phi_{z}(r,\Gth,0)=\phi_{\Gth}(r,\Gth,L)=\phi_{z}(r,\Gth,L)=0\},
\]
which gives
\begin{equation}
\label{Breather}
V_{h}=\{\BGf\in W^{1,2}(\CC_{h};\mathbb R^3):
\phi_{\Gth}(r,\Gth,0)=\phi_{z}(r,\Gth,0)=\phi_{\Gth}(r,\Gth,L)=\phi_{z}(r,\Gth,L)=0\}.
\end{equation}
These \bc s allow the shell to ``breathe'', since the radial
displacements are not prescribed at either end. In our notation the dependence
on $L$ will be consistently
suppressed, while the essential dependence on $h$ will be emphasized.

We observe that during buckling the Cauchy-Green strain tensor
$\BC=\BF^{T}\BF$ is close to the identity. Therefore, considering the energy
which is quadratic in $\BE=(\BC-\BI)/2$ should capture all the effects
associated with buckling. Hence, we assume, for the purposes of exhibiting the
explicit form of the trivial branch, that in the vicinity of the identity
matrix the energy density has the Saint Venant-Kirchhoff form:
\[
W(\BF)=\hf(\SFL_{0}\BE,\BE),\qquad\BE=\hf(\BF^{T}\BF-\BI).
\]
where the elastic tensor $\SFL_{0}$ is
isotropic.
We study stability of the homogeneous trivial branch
$\By(\Bx;h,\Gl)$ given in cylindrical coordinates by
\begin{equation}
  \label{trbr}
y_{r}=(a(\Gl)+1)r,\qquad y_{\Gth}=0,\qquad y_{z}=(1-\Gl)z.
\end{equation}
We compute, using the formula
\begin{equation}
\label{gradient}
\nabla \BGf=
\begin{bmatrix}
\phi_{r,r} & \dfrac{\phi_{r,\theta}-\phi_\theta}{r} & \phi_{r,z}\\
\phi_{\theta,r} & \dfrac{\phi_{\theta,\theta}+\phi_r}{r} & \phi_{\theta,z}\\
\phi_{z,r} & \dfrac{\phi_{z,\theta}}{r} & \phi_{z,z}\\
\end{bmatrix}.
\end{equation}
for the gradient of the vector field
$\BGf=\phi_r\Be_r+\phi_\theta\Be_\theta+\phi_z\Be_z$ in cylindrical coordinates
\[
\BF=\Grad\By=\left[
  \begin{array}{ccc}
    1+a & 0 &0\\
    0 & 1+a& 0\\
    0 & 0& 1-\Gl
  \end{array}
\right],\quad
\BE=\left[
  \begin{array}{ccc}
    a+\frac{a^{2}}{2} & 0 &0\\
    0 & a+\frac{a^{2}}{2}& 0\\
    0 & 0& \frac{\Gl^{2}}{2}-\Gl
  \end{array}
\right]
\]
Then we compute $\BP=\BF(\SFL_{0}\BE)$, and the traction-free condition
$\BP\Be_{r}=\Bzr$ on the lateral boundary leads to the expression for $a(\Gl)$:
\[
a(\Gl)=\sqrt{1+2\nu\Gl-\nu\Gl^{2}}-1,
\]
where $\nu$ is the Poisson's ratio for $\SFL_{0}$.
We now see that the fundamental assumptions (\ref{fundass}) and
(\ref{reglbda}) are satisfied, since
the trivial branch does not depend on $h$ explicitly. We compute
\begin{equation}
  \label{perfect.stress.tensor}
\BGs_{h}=-E\tns{\Be_{z}},
\end{equation}
where $E$ is the Young's modulus.
The compression tensor $\Tld{\BGs}_{h}$ defined in (\ref{compten}) is given by
\[
\Tld{\BGs}_{h}=-E\left[
  \begin{array}{ccc}
    1 & 0 & 0\\
    0 & 1 & 0\\
    0 & 0 & 0
  \end{array}
\right].
\]
We see that the compression tensor is degenerate. This degeneracy in the compression
tensor is one of the factors contributing to the sensitivity of the critical
load to imperfections.

\subsection{Scaling of the critical load}
\label{sec:scaling.lows}
\begin{theorem}
  \label{th:pcl}
Suppose that $\BGs_{h}$ is given by (\ref{perfect.stress.tensor}). Then there
exist constants $c>0$ and $C>0$ depending only on $L$ and the elastic moduli,
such that
\begin{equation}
 \label{perfect.lambda.hat.upper}
 ch\le\Hat{\lambda}(h)\leq Ch.
 \end{equation}
\end{theorem}
\begin{proof}
Observe that
$$
\mathfrak{C}_{h}(\BGf)=\int_{\CC_{h}}(\BGs_{h},\Grad\BGf^{T}\Grad\BGf)d\Bx=-E(\|\Gf_{r,z}\|^2+\|\Gf_{\Gth,z}\|^2+\|\Gf_{z,z}\|^2),
$$
and there exist constants $\Ga>0$ and $\Gb>0$ (depending only on the elastic
moduli) such that
\[
\Ga\|e(\BGf)\|^{2}\le\mathfrak{S}_{h}(\BGf)\le\Gb\|e(\BGf)\|^{2}.
\]
Thus, in order to compute the scaling of $\Hat{\Gl}(h)$, given by (\ref{clin})
and verify conditions of Theorem~\ref{th:crit} we need to estimate the Korn
constant $K(V_{h})$, as well as the norms of gradient components
$\|\Gf_{r,z}\|^2$, $\|\Gf_{\Gth,z}\|^2$ and $\|\Gf_{z,z}\|^2$ in terms of
$\|e(\BGf)\|$. This was accomplished in our companion paper \cite{grha14}. The
desired estimates are stated in the following lemma.
\begin{lemma}[Korn-type inequalities]
\label{lem:KI}
There exist constants $C(L),c(L)>0$ depending only on $L$ such that
\begin{equation}
  \label{KI}
 c(L)h^{3/2}\leq K(V_{h})\leq C(L)h^{3/2}.
\end{equation}
 \begin{equation}
  \label{thetaz}
  \|\phi_{\Gth,z}\|^{2}\le\frac{C(L)}{\sqrt{h}}\|e(\BGf)\|^{2},
\end{equation}
\begin{equation}
  \label{rz}
   \|\phi_{r,z}\|^{2}\le\frac{C(L)}{h}\|e(\BGf)\|^{2}.
\end{equation}
Moreover, the powers of $h$ in the inequalities (\ref{KI})--(\ref{rz}) are optimal,
achieved \emph{simultaneously} by the ansatz
\begin{equation}
\begin{cases}
\label{ansatz0}
\phi^{h}_{r}(r,\Gth,z)=&-W_{,\eta\eta}\left(\frac{\Gth}{\sqrt[4]{h}},z\right)\\[2ex]
\phi^{h}_{\Gth}(r,\Gth,z)=&r\sqrt[4]{h}W_{,\eta}\left(\frac{\Gth}{\sqrt[4]{h}},z\right)+
\frac{r-1}{\sqrt[4]{h}}W_{,\eta\eta\eta}\left(\frac{\Gth}{\sqrt[4]{h}},z\right),\\[2ex]
\phi^{h}_{z}(r,\Gth,z)=&(r-1)W_{,\eta\eta z}\left(\frac{\Gth}{\sqrt[4]{h}},z\right)
-\sqrt{h}W_{,z}\left(\frac{\Gth}{\sqrt[4]{h}},z\right),
\end{cases}
\end{equation}
 for any smooth compactly supported function $W(\eta,z)$ on
 $(-1,1)\times(0,L)$, with the understanding that the function $\BGf^{h}(\Gth,z)$ is extended
 $2\pi$-periodically in $\Gth\in\bb{R}$.
\end{lemma}
Adding inequalities (\ref{thetaz}), (\ref{rz}) and an obvious inequality
$\|\phi_{z,z}\|^{2}\le\|e(\BGf)\|^{2}$ we obtain
\begin{equation}
  \label{Rest}
-\mathfrak{C}_{h}(\BGf)\le Ch\mathfrak{S}_{h}(\BGf).
\end{equation}
The power of $h$ in (\ref{Rest}) is optimal, achieved by the ansatz
(\ref{ansatz0}). Hence, the estimates (\ref{perfect.lambda.hat.upper}) are
proved.
\end{proof}
 We remark that the upper bound in (\ref{perfect.lambda.hat.upper}) implies
 that condition (\ref{sufcond}) in Theorem~\ref{th:crit} is satisfied. But
 then $\Hat{\lambda}(h)$ is the buckling load in the sense of
 Definition~\ref{def:asymbm}.
\begin{figure}[t]
  \centering
  \includegraphics[scale=0.5]{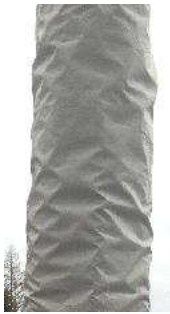}
  \caption{Yoshimura buckling pattern on the umbrella cover at the
    Mathematisches Forschungsinstitut, Oberwolfach, Germany. Photo by Antonio
    DeSimone.}
  \label{fig:mfocover}
\end{figure}
\begin{remark}
  We remark that the scaling of the critical strain $\Gl^{*}(h)\sim h$ implies
  that the elastic energy stored in the critically strained cylinder is of
  order $h^{3}$, since the stress remains proportional to the strain at the
  onset of buckling. Thus, the $\GG$-limit theorem from \cite{fjmm07}
  applies. However, that theorem misses the structure of low energy
  sequences, since the set of $W^{2,2}$ isometries of the cylindrical surface
  consists of rigid motions, and the limiting energy is zero. The non-trivial
  isometries are non-smooth (Lipschitz), given by the Yoshimura buckling
  pattern \cite{yosh55} (see Figure~\ref{fig:mfocover}), which seems to be
  captured by some of the theoretical buckling modes in \cite{grha15}.
\end{remark}
\subsection{Scaling instability}
\label{sub:scins}
In this section we exhibit scaling instability of the critical load under
imperfections of load and shape. The discussion of shape imperfections here is
not rigorous, since the key Korn and Korn-type inequalities from \cite{grha14}
are rigorously proved for perfectly circular cylindrical shells. However, once
the necessary technical inequalities are established for an imperfect shell,
its critical load can be estimated in a definitive and rigorous way, following
the same strategy, as for a perfect shell.

\subsubsection{Imperfections of load}
Consider a perfect isotropic circular cylindrical shell undergoing a
compressive deformation satisfying the \bc s (\ref{bc}), which are perturbed
\emph{arbitrarily}, but only at $z=L$. We also assume that the modified \bc s
do not violate the trivial branch regularity assumptions in
Definition~\ref{def:trbr}. Let us make an additional assumption that the
family of Lipschitz functions $\Bu^{h}$ from Definition~\ref{def:trbr} depends
regularly on $r$ and $h$. This means that
\begin{equation}
  \label{rhreg}
  \Bu^{h}(r,\Gth,z)\approx\Tld{\Bu}^{h}(r,\Gth,z)=\Bu^{0}(\Gth,z)+(r-1)\Bu^{1}(\Gth,z)+\frac{(r-1)^{2}}{2}\Bu^{2}(\Gth,z),
\end{equation}
understood in the following sense:
\[
\lim_{h\to 0}\Bu^{h}=\lim_{h\to 0}\Tld{\Bu}^{h}=\Bu^{0},\qquad
\lim_{h\to 0}\Grad\Bu^{h}=\lim_{h\to 0}\Grad\Tld{\Bu}^{h}=\lim_{h\to 0}\Grad(\Bu^{0}+(r-1)\Bu^{1}),
\]
\[
\lim_{h\to 0}\Grad\Grad\Bu^{h}=\lim_{h\to 0}\Grad\Grad\Tld{\Bu}^{h}
\]
where the first two limits are understood in the a.e. sense, while the
last limit is understood in the sense of distributions.

According to the formula (\ref{clin}),
the buckling load depends only on
\[
\BGs^{0}(\Gth,z)=\lim_{h\to 0}\BGs_{h}(r,\Gth,z)=\lim_{h\to 0}\SFL_{0}e(\Bu^{h}).
\]
Under these assumptions generic load imperfections at $z=L$ may involve
several functions of $\Gth$.  We will show now that somewhat surprisingly,
the regularity assumptions (\ref{rhreg}) guarantee that the set of possible
limits $\BGs^{0}(\Gth,z)$ depends only on 2 scalar parameters.
\begin{theorem}
  \label{th:loadimp}
Suppose that $\Bu^{h}(r,\Gth,z)$ depends on $r$ and $h$ regularly, in the
sense of (\ref{rhreg}). Suppose further that
\begin{itemize}
\item[(i)] $\Div(\SFL_{0}e(\Bu^{h}))=\Bzr$,
\item[(ii)] $\BGs_{h}\Be_{r}=\Bzr$ at $r=1\pm h/2$, where $\BGs_{h}=\SFL_{0}e(\Bu^{h})$,
\item[(iii)] $u^{h}_{z}(r,\Gth,0)=u^{h}_{\Gth}(r,\Gth,0)=0$.
\end{itemize}
Then there exist two constants $s$ and $t$, such that
\begin{equation}
  \label{uhrep}
  u^{0}_{r}+(r-1)u_{r}^{1}=-\frac{t\nu}{E}r,\quad
u^{0}_{\Gth}+(r-1)u_{\Gth}^{1}=\frac{2(1+\nu)s}{E}rz,\quad
u^{0}_{z}+(r-1)u_{z}^{1}=\frac{t}{E}z,
\end{equation}
and consequently
\begin{equation}
  \label{sigma0}
  \BGs^{0}=
  \begin{bmatrix}
    0 & 0 & 0\\
    0 & 0 & s\\
    0 & s & t
  \end{bmatrix}
\end{equation}
\end{theorem}
\begin{proof}
By the assumptions of regularity (\ref{rhreg}) and by condition (i) we have
\begin{equation}
  \label{div0}
\lim_{h\to 0}\Div\BGs_{h}=\lim_{h\to 0}\Div(\BGs^{0}(\Gth,z)+(r-1)\BGs^{1}(\Gth,z))=0,
\end{equation}
where
\begin{equation}
  \label{ss}
\BGs^{0}=\lim_{h\to 0}\SFL_{0}e(\Bu^{0}+(r-1)\Bu^{1}),\qquad
\BGs^{1}=\lim_{h\to 0}\SFL_{0}e\left(\Bu^{1}+\frac{r-1}{2}\Bu^{2}\right).
\end{equation}
Passing to the limit as $h\to 0$ in (\ref{div0}), we obtain
\begin{equation}
  \label{limequil}
  \begin{cases}
  \Gs^{1}_{rr}+\Gs^{0}_{r\Gth,\Gth}+\Gs^{0}_{rr}-\Gs^{0}_{\Gth\Gth}+\Gs^{0}_{rz,z}=0,\\
\Gs^{1}_{r\Gth}+\Gs^{0}_{\Gth\Gth,\Gth}+2\Gs^{0}_{r\Gth}+\Gs^{0}_{\Gth z,z}=0,\\
\Gs^{1}_{rz}+\Gs^{0}_{\Gth z,\Gth}+\Gs^{0}_{rz}+\Gs^{0}_{zz,z}=0.
  \end{cases}
\end{equation}
The traction-free \bc s $\BGs_{h}\Be_{r}=\Bzr$ at $r=1\pm h/2$ imply that
\[
\BGs^{0}(\Gth,z)\Be_{r}=\BGs^{1}(\Gth,z)\Be_{r}=\Bzr
\]
for all $(\Gth,z)\in\bb{T}\times(0,L)$. Substituting these equations into
(\ref{limequil}) we obtain
\[
\Gs^{0}_{\Gth\Gth}=0,\quad\Gs^{0}_{\Gth z,z}=0,\qquad\Gs^{0}_{\Gth z,\Gth}+\Gs^{0}_{zz,z}=0.
\]
Solving these equations we obtain
\begin{equation}
  \label{linstr}
\BGs^{0}(\Gth,z)=
  \begin{bmatrix}
    0 & 0 & 0\\
    0 & 0 & s(\Gth)\\
    0 & s(\Gth) & t(\Gth)-zs'(\Gth)
  \end{bmatrix}.
\end{equation}
for some functions $s(\Gth)$ and $t(\Gth)$. The first equation in (\ref{ss})
can now be written as
\begin{equation}
  \label{ueq}
  \begin{cases}
  u^{1}_{\Gth}=u^{0}_{\Gth}-u^{0}_{r,\Gth},\quad u^{1}_{z}=-u^{0}_{r,z},\quad
u^{0}_{\Gth,z}+u^{0}_{z,\Gth}=\frac{2(1+\nu)}{E}s(\Gth),\\[2ex]
u^{1}_{r}=-\frac{\nu}{1-\nu}(u^{0}_{r}+u^{0}_{\Gth,\Gth}+u^{0}_{z,z}),\quad
u^{0}_{r}+u^{0}_{\Gth,\Gth}+\frac{\nu}{1-\nu}(u^{1}_{r}+u^{0}_{z,z})=0,\\[2ex]
u^{0}_{z,z}+\frac{\nu}{1-\nu}(u^{1}_{r}+u^{0}_{r}+u^{0}_{\Gth,\Gth})=
\frac{(1+\nu)(1-2\nu)}{E(1-\nu)}(t(\Gth)-zs'(\Gth)).
  \end{cases}
\end{equation}
Solving these equations subject to the conditions
\[
u^{0}_{z}(\Gth,0)=u^{0}_{\Gth}(\Gth,0)=u^{1}_{z}(\Gth,0)=u^{1}_{\Gth}(\Gth,0)=0
\]
we conclude that the functions $s(\Gth)$ and $t(\Gth)$ have to be
constant\footnote{We note that $s=$constant is a consequence of
 $u^{1}_{z}(\Gth,0)=0$, while $t=$constant is a consequence of
 $u^{1}_{\Gth}(\Gth,0)=0$.} and that the formulas (\ref{uhrep}) hold. The
formula (\ref{sigma0}) follows from (\ref{linstr}).
\end{proof}
Thus, the effect of generic imperfections of load on $\BGs^{0}$ may manifest
themselves only through a small perturbation of $(z,z)$-component, and the
appearance of a small constant $(\Gth z)$-component. In order to prove
rigorously that imperfections of shape can indeed result in $\BGs^{0}$ of the
form (\ref{sigma0}) with $s\not=0$ we need to exhibit a fully non-linear
trivial branch satisfying all our assumptions and leading to
(\ref{sigma0}). It is clear that the non-linear trivial branch satisfying the
perturbed \bc s may no longer be homogeneous. This prevents us from
exhibiting it explicitly for the Saint Venant-Kirchhoff energy, as in
(\ref{trbr}). However, for an incompressible Mooney-Rivlin material the desired
non-linear trivial branch can be computed explicitly (see
Appendix~\ref{app:nltb}).

It is an important feature of our approach, that in order to compute the
asymptotics of the buckling load and the buckling mode we do not need to know
the non-linear trivial branch explicitly. (We only need to know that the
linearly elastic trivial branch, in the sense of Definition~\ref{def:trbr},
exists.) The desired asymptotics is given by Theorem~\ref{th:crit} in terms of
the solution $\Bu^{h}$ of the equations of \emph{linear elasticity}. In order to obtain
$\BGs^{0}$ of the form (\ref{sigma0}) we observe that
\begin{equation}
  \label{lintor}
u^{h}_{r}=\nu r,\qquad u^{h}_{\Gth}=\Ge rz,\qquad u^{h}_{z}=-z,
\end{equation}
solves
\begin{equation}
  \label{semidef}
  \begin{cases}
    \Div(\SFL_{0}e(\Bu^{h}))=\Bzr,&\text{ in }\CC_{h},\\
    \BGs_{h}\Be_{r}=\Bzr,&r=1\pm\frac{h}{2},\\
    u^{h}_{z}=u^{h}_{\Gth}=0,&z=0,
  \end{cases}
\end{equation}
resulting in
\begin{equation}
  \label{impstr}
  \BGs_{h}=\left[
  \begin{array}{ccc}
    0 & 0 & 0\\
    0 & 0 & \dfrac{\Ge Er}{2(\nu+1)}\\
    0 & \dfrac{\Ge Er}{2(\nu+1)} & -E
  \end{array}
\right],\qquad
\BGs^{0}=\left[
  \begin{array}{ccc}
    0 & 0 & 0\\
    0 & 0 & \dfrac{\Ge E}{2(\nu+1)}\\
    0 & \dfrac{\Ge E}{2(\nu+1)} & -E
  \end{array}
\right].
\end{equation}
In this explicit solution the imperfections of load are described by a single
small, in absolute value, parameter $\Ge$. This specific representation of
$\BGs^{0}$ is, nevertheless, generic for arbitrary imperfections of load at
$z=L$, according to Theorem~\ref{th:loadimp}. Similarly to
Theorem~\ref{th:pcl}, the formulas (\ref{impstr}) determine the scaling of the
critical load with $h$, which, for every fixed $\Ge$, is \emph{different} from
(\ref{perfect.lambda.hat.upper}).
\begin{theorem}
  \label{th:licls}
  Suppose that $\BGs^{0}$ is given by (\ref{impstr}). Then there are positive
  constants $c$ and $C$, depending only on $L$ and the elastic moduli, such that
  \begin{equation}
    \label{impldest}
\frac{ch^{5/4}}{\Ge+h^{1/4}}\le\Hat{\Gl}(h)\le\frac{Ch^{5/4}}{\Ge+h^{1/4}},
  \end{equation}
when $h$ and $\Ge$ are sufficiently small.
\end{theorem}
\begin{proof}
Let
\[
\mathfrak{C}_{h}^{0}(\BGf)=\int_{\CC_{h}}(\BGs^{0},\Grad\BGf^{T}\Grad\BGf)d\Bx.
\]
We first prove the lower bound on $\Hat{\Gl}(h)$ by observing that
\[
-\mathfrak{C}_{h}^{0}(\BGf)=E(\|(\Grad\BGf)_{rz}\|^{2}+\|(\Grad\BGf)_{\Gth z}\|^{2})-
\dfrac{\Ge E}{2(\nu+1)}((\Grad\BGf)_{rz},(\Grad\BGf)_{r\Gth})+R_{h}(\BGf),
\]
where $(f,g)$ denotes the inner product in $L^{2}(\CC_{h})$ and
\[
R_{h}(\BGf)=E\|(\Grad\BGf)_{zz}\|^{2}-
\dfrac{\Ge E}{2(\nu+1)}\{((\Grad\BGf)_{\Gth z},(\Grad\BGf)_{\Gth\Gth})+
(\Grad\BGf)_{z\Gth},(\Grad\BGf)_{zz})\}.
\]
Then, for every $\BGf\in V_{h}$
\[
|R_{h}|\le C(\|e(\BGf)\|^{2}+\|e(\BGf)\|\|\Grad\BGf\|)\le
\frac{C\|e(\BGf)\|^{2}}{\sqrt{K(V_{h})}}.
\]
Let
\[
\Tld{\mathfrak{R}}(h,\BGf)=\frac{\mathfrak{S}_{h}(\BGf)}
{E(\|(\Grad\BGf)_{rz}\|^{2}+\|(\Grad\BGf)_{\Gth z}\|^{2})-
\dfrac{\Ge E}{2(\nu+1)}((\Grad\BGf)_{rz},(\Grad\BGf)_{\Gth r})}.
\]
Then
\[
\left|\nth{\Tld{\mathfrak{R}}(h,\BGf)}-\nth{\mathfrak{R}(h,\BGf)}\right|\le
\frac{C}{\sqrt{K(V_{h})}},
\]
and hence,
by Theorem~\ref{th:Bequivalence}, the pair $(\Tld{\mathfrak{R}}(h,\BGf),V_{h})$
is buckling equivalent to the pair $(\mathfrak{R}(h,\BGf),V_{h})$.
By Lemma~\ref{lem:KI} we obtain,
applying the Cauchy-Schwarz inequality,
$$|((\Grad\BGf)_{rz},(\Grad\BGf)_{\Gth r})|\leq\|(\Grad\BGf)_{rz}\|\|\Grad\BGf\|\leq
\frac{C\|e(\BGf)\|}{\sqrt{h}}K(V_{h})\|e(\BGf)\|\le\frac{C\|e(\BGf)\|^{2}}{h^{5/4}}.$$
Applying Lemma~\ref{lem:KI} and (\ref{Lcoerc}) we obtain
\begin{equation}
  \label{epsbd}
\Tld{\mathfrak{R}}(h,\BGf)\ge\frac{\Ga_{\SFL_{0}}\|e(\BGf)\|^{2}}
{C\|e(\BGf)\|^{2}(h^{-1}+h^{-1/2}+\Ge h^{-5/4})}\ge \frac{Ch^{5/4}}{\Ge+h^{1/4}},
\end{equation}
To obtain an upper bound on the critical load, we use test functions $\BGf^{h}$ given by
(\ref{ansatz0}) in the estimate
\[
\Hat{\Gl}(h)\le C\Tld{\mathfrak{R}}(h,\BGf^{h}).
\]
Using the explicit formulas (\ref{ansatz0}) for $\BGf^{h}$ we compute
\begin{equation}
  \label{impker}
\lim_{h\to 0}\nth{h}((\Grad\BGf^{h})_{rz},(\Grad\BGf^{h})_{\Gth r})
=\int_{0}^{2\pi}\int_{0}^{L}W_{,\eta\eta\eta}(\eta,z)W_{,\eta\eta  z}(\eta,z)d\eta dz.
\end{equation}
By Lemma~\ref{lem:KI}, in order to prove the upper bound in (\ref{impldest})
we only need to exhibit a fixed compactly supported function $W(\eta,z)$, such that
the \rhs\ in (\ref{impker}) is non-zero. This is done by choosing two arbitrary
non-zero compactly supported functions $\phi(\eta)$ and $\psi(z)$ and setting
\[
W(\eta,z)=\phi(\eta)\psi'(z)+\phi'(\eta)\psi(z).
\]
Then
\begin{multline*}
W_{,\eta\eta\eta}W_{,\eta\eta  z}=\nth{4}(\psi'(z)^{2})'(\phi''(\eta)^{2})'+
(\phi'''(\eta)^{2})'(\psi(z)^{2})'+(\phi'''(\eta)\phi''(\eta))'\psi(z)\psi''(z)\\
-\phi'''(\eta)^{2}(\psi(z)\psi'(z))'+2\phi'''(\eta)^{2}\psi'(z)^{2}.
\end{multline*}
This shows that
\[
\int_{0}^{2\pi}\int_{0}^{L}W_{,\eta\eta\eta}W_{,\eta\eta  z}d\eta dz=
2\int_{0}^{2\pi}\int_{0}^{L}\phi'''(\eta)^{2}\psi'(z)^{2}d\eta dz>0.
\]
\end{proof}
From the estimates (\ref{impldest}) we see that in order for the scaling
$h^{5/4}$ to be experimentally significant, $|\Ge|$ must be much larger than
$h^{1/4}$. This is unlikely for the typical values of $h\approx
10^{-4}$. Nevertheless, Theorem~\ref{th:licls} (together with
Appendix~\ref{app:nltb}) demonstrates rigorously that axially compressed
cylindrical shells exhibit scaling instability under imperfections of
load. We can also view this result as a strong indication that it is
the imperfections of shape that are largely responsible for the discrepancy
between the theory and experiment.

\subsubsection{Imperfection of shape}
In the case of shape imperfections our Korn inequalities for gradient and
gradient components, strictly speaking, cannot be applied, since the domain is
no longer $\CC_{h}$. In this case we conjecture that for some shape
imperfections, such as small localized dents the trivial branch would still
exist and satisfy out assumptions (\ref{fundass}), while the Korn constant
retains its $h^{3/2}$ asymptotics. While the arguments below are not exactly
rigorous, we believe that they do shed new light on the question of rigorous
estimation of the critical load for an imperfect cylindrical shell. The key
insight achieved in the foregoing analysis is that the reason for the
difference in scaling laws of the critical strain and the Korn constant is the
structure of the stress in the trivial branch (which in a perfect axially
compressed cylinder has only $zz$-component that is non-zero). The failure of
the imperfections of load to modify this structure in a significant way (see
Theorem~\ref{th:loadimp}) is due to the traction-free \bc s on the lateral surfaces of
the shell. This observation leads to the idea that if the shell is ``dented'',
the normal to the lateral surface may undergo a non-negligible change in a
small region. To model this mathematically we assume that the dented
cylindrical shell is given by
\[
\Tld{\CC}_{h}=\left\{(r,\Gth,z):\Gth\in\bb{T},\ z\in[0,L],1+
\Ge^{2}\rho\left(\frac{\Gth}{\Ge},\frac{z-z_{0}}{\Ge}\right)-\frac{h}{2}\le r\le 1+
\Ge^{2}\rho\left(\frac{\Gth}{\Ge},\frac{z-z_{0}}{\Ge}\right)+\frac{h}{2}\right\},
\]
where the function $\rho(\eta,\Gz)$ is compactly supported on a unit ball in
$\bb{R}^{2}$, where $\rho(\Gth/\Ge,(z-z_{0})/\Ge)$ is meant for $\Gth\in[-\pi,\pi]$ and
is understood as a $2\pi$-periodic function. We assume that $\Ge=\Ge(h)\to 0$,
as $h\to 0$ and $h/\Ge(h)\to 0$, as $h\to 0$. For the ``proof-of-concept''
demonstration we assume, without proof, that the linear stress in the trivial
branch can be written as
\begin{equation}
  \label{exp}
  \BGs^{h}(r,\Gth,z)=\BGs^{h}_{p}+\Tld{\BGs}^{h}(\Gth,z)+(r-1)\BGt^{h}(\Gth,z)+o(h),
\end{equation}
where $\BGs^{h}_{p}$ is the stress in the perfect shell, given by
(\ref{perfect.stress.tensor}). We assume that $\Tld{\BGs}^{h}=O(1)$ and
$\BGt^{h}=O(1)$, as $h\to 0$, while
\begin{equation}
  \label{Dexp}
  \Grad\BGs^{h}(r,\Gth,z)=\Grad(\BGs^{h}_{p}+\Tld{\BGs}^{h}(\Gth,z)+(r-1)\BGt^{h}(\Gth,z))+o(1).
\end{equation}
The normal to the traction-free surface of the imperfect cylinder
$\Tld{\CC}_{h}$ is now
\[
\BN_{h}=\Be_{r}-\Ge(\rho_{,\eta}\Be_{\Gth}+\rho_{,\Gz}\Be_{z}).
\]
According to (\ref{exp}) we must have
\begin{equation}
  \label{trfrimp}
  \Tld{\BGs}^{h}\BN_{h}+\BGs^{h}_{p}\BN_{h}=o(h),\qquad\BGt^{h}\BN_{h}=o(1).
\end{equation}
In components this implies
\[
\begin{cases}
\Tld{\Gs}^{h}_{rr}=\Ge(\rho_{,\eta}\Tld{\Gs}^{h}_{r\Gth}+\rho_{,\Gz}\Tld{\Gs}^{h}_{rz})+o(h),\\
\Tld{\Gs}^{h}_{r\Gth}=\Ge(\rho_{,\eta}\Tld{\Gs}^{h}_{\Gth\Gth}+\rho_{,\Gz}\Tld{\Gs}^{h}_{\Gth z})+o(h),\\
\Tld{\Gs}^{h}_{rz}=\Ge(-E\rho_{,\Gz}+\rho_{,\eta}\Tld{\Gs}^{h}_{\Gth z}+\rho_{,\Gz}\Tld{\Gs}^{h}_{zz})+o(h),
\end{cases}\qquad
\begin{cases}
  \tau^{h}_{rr}=o(1),\\
  \tau^{h}_{r\Gth}=o(1),\\
  \tau^{h}_{rz}=o(1).
\end{cases}
\]
The balance equations $\Div\BGs^{h}=\Bzr$ then become
\begin{equation}
  \label{impbal}
\begin{cases}
  \Ge\dif{}{\Gth}(\rho_{,\eta}\Tld{\Gs}^{h}_{\Gth\Gth}+\rho_{,\Gz}\Tld{\Gs}^{h}_{\Gth
    z})-\Tld{\Gs}^{h}_{\Gth\Gth}+\Ge\dif{}{z}(-E\rho_{,\Gz}+\rho_{,\eta}\Tld{\Gs}^{h}_{\Gth
    z}+\rho_{,\Gz}\Tld{\Gs}^{h}_{zz})=o(1),\\
\Tld{\Gs}^{h}_{\Gth\Gth,\Gth}+\Tld{\Gs}^{h}_{\Gth z,z}=o(1),\\
\Tld{\Gs}^{h}_{\Gth z,\Gth}+\Tld{\Gs}^{h}_{zz,z}=o(1).
\end{cases}
\end{equation}
At this point we abandon any semblance of rigor and set the \rhs s in
(\ref{impbal}) to zero and assume that
\[
\Tld{\BGs}^{h}=\Hat{\BGs}\left(\frac{\Gth}{\Ge},\frac{z-z_{0}}{\Ge}\right).
\]
The last two equations in (\ref{impbal}) then implies that
\[
\Tld{\Gs}^{h}_{\Gth\Gth}=s_{,\Gz\Gz}\left(\frac{\Gth}{\Ge},\frac{z-z_{0}}{\Ge}\right),\qquad
\Tld{\Gs}^{h}_{\Gth z}=-s_{,\eta\Gz}\left(\frac{\Gth}{\Ge},\frac{z-z_{0}}{\Ge}\right),\qquad
\Tld{\Gs}^{h}_{zz}=s_{,\eta\eta}\left(\frac{\Gth}{\Ge},\frac{z-z_{0}}{\Ge}\right).
\]
The first equation in (\ref{impbal}) becomes
\begin{equation}
  \label{PDE}
\rho_{,\eta\eta}s_{,\Gz\Gz}+s_{,\eta\eta}\rho_{,\Gz\Gz}-2\rho_{,\eta\Gz}s_{,\eta\Gz}=
s_{,\Gz\Gz}+E\rho_{,\Gz\Gz}.
\end{equation}
If we assume that $\rho_{,\eta\eta}(\eta,\Gz)$ and $\rho_{,\eta\Gz}$ are
uniformly small (i.e. the dent is localized significantly more in the $z$
direction than in $\Gth$), then $s\approx-E\rho$.  In order to trigger the
mode of instability with the critical strain scaling like the Korn constant
$\Gl(h)\sim h^{3/2}$ we require $\Gs^{h}_{\Gth\Gth}<-\Ga<0$ in a \nbh\ of a
point $(0,z_{0})$, i.e.  $\rho_{,\Gz\Gz}(0,0)>0$. This can be achieved only on
``inward dents''.  In general, we assume that there exists a decaying at
infinity solution $s(\eta,\Gz)$ of (\ref{PDE}), such that
$s_{,\Gz\Gz}(0,0)<0$.

In conclusion we note that the exponents $5/4=1.25$, associated with load
imperfections and $3/2=1.5$, associated with imperfections of shape are close
to the upper and lower limits of experimentally determined behavior of the
buckling load, respectively, \cite{call2000,hlp03}. We also note that the
observed buckling load of
the real imperfect structure may be further affected by the
subcritical nature of the respective bifurcations (see \cite{buha79} for a
lucid explanation why).

\medskip

\noindent\textbf{Acknowledgments.}  We are grateful to Eric Clement, Stefan
Luckhaus, Mark Peletier and Lev Truskinovsky for insightful comments and questions.
This material is based upon work supported by the National Science
Foundation under Grants No. 1008092.

\appendix

\section{Non-linear trivial branch for an  incompressible Mooney-Rivlin material.}
\setcounter{equation}{0}
\label{app:nltb}
Consider an incompressible Mooney-Rivlin type material with strain energy
function
\[
W(\BF)=\frac{E}{6}(|\BF|^{2}-3),\qquad\det\BF=1.
\]
We are looking for a trivial branch in a cylindrical shell, given in
cylindrical coordinates by
\begin{equation}
  \label{tbansatz}
y_{r}=\psi(r)\cos(\Ga z),\qquad y_{\Gth}=\psi(r)\sin(\Ga z),\qquad
y_{z}=(1-\Gl)z.
\end{equation}
It is expected that $\psi(r)$ also depends on $\Ga$, $\Gl$ and $h$. When
$\Ga=0$ we expect that $\psi(r)$ will reduce to $(a(\Gl)+1)r$, as in
(\ref{trbr}). We remark that, in principle, the ansatz (\ref{tbansatz}) should
also work for compressible materials, except the resulting non-linear second order ODE
for $\psi(r)$ cannot be solved explicitly. We compute
\[
\det(\Grad\By)=(1-\Gl)\psi'(r)\frac{\psi(r)}{r}.
\]
For an incompressible material we must have $\det(\Grad\By)=1$, and hence
\begin{equation}
  \label{psi}
\psi(r)=\sqrt{\frac{r^{2}}{1-\Gl}+\Gb}
\end{equation}
for some $\Gb>-1$.

The Piola-Kirchhoff stress function is
\[
\BP(\BF)=\frac{E}{3}\left(\BF-\frac{3\hat{p}}{E}\cof(\BF)\right),
\]
where the Lagrange multiplier $\hat{p}$ plays the role of pressure.
For $\By$, given by (\ref{tbansatz}) and $\BF=\Grad\By$ we compute
\[
\BF^{T}\BF=
\begin{bmatrix}
  (\psi'(r))^{2} & 0 & 0\\
  0 & \frac{\psi(r)^{2}}{r^{2}} & \frac{\Ga\psi(r)^{2}}{r}\\
  0 & \frac{\Ga\psi(r)^{2}}{r} & \Ga^{2}\psi(r)^{2} + (1-\Gl)^{2}
\end{bmatrix}.
\]
The traction-free condition $\BP\Be_{r}=\Bzr$ on $r=1\pm h/2$ can be written
as
\[
\BF^{T}\BF\Be_{r}=p\Be_{r},\quad r=1\pm\frac{h}{2},\qquad p=3\hat{p}/E.
\]
The formula for $\BF^{T}\BF$, together with $\det\BF=1$, implies that
\begin{equation}
  \label{tfbc}
p(r,\Gth,z)=(\psi'(r))^{2},\quad r=1\pm\frac{h}{2}.
\end{equation}
This suggests that it is reasonable to look for the trivial branch for which
the function $p(r,\Gth,z)$ depends only on $r$.
Under this assumption we compute
\[
\frac{3}{E}\BP=
\begin{bmatrix}
  s_{1}(r)\cos(\Ga z) & -s_{2}(r)\sin(\Ga z) & -s_{3}(r)\sin(\Ga z)\\
  s_{1}(r)\sin(\Ga z) & s_{2}(r)\cos(\Ga z) & s_{3}(r)\cos(\Ga z)\\
  0 & q_{1}(r) & q_{2}(r)
\end{bmatrix},
\]
where
\[
s_{1}=\psi'-\frac{p}{\psi'},\quad s_{2}=\frac{\psi}{r}-\frac{rp}{\psi},\quad
s_{3}=\Ga\psi,\quad q_{1}=\frac{\Ga rp}{1-\Gl},\quad q_{2}=1-\Gl-\frac{p}{1-\Gl}.
\]
It follows that $\Div\BP=\Bzr$ results in a single ODE for $p(r)$:
\begin{equation}
  \label{ode}
  (rs_{1})'=s_{2}+\Ga rs_{3}.
\end{equation}
Substituting (\ref{psi}) for $\psi(r)$ into (\ref{ode}) and solving for $p(r)$ we obtain
\[
p(r)=\nth{2(1-\Gl)}\left(
\ln\left(\frac{1}{1-\Gl}+\frac{\Gb}{r^{2}}\right)-r^{2}\Ga^{2}
-\frac{\Gb(1-\Gl)}{r^{2}+\Gb(1-\Gl)}+\Gg\right).
\]
The traction-free \bc s (\ref{tfbc}) become
\[
\frac{r^{2}}{r^{2}+\Gb(1-\Gl)}=\ln\left(\frac{1}{1-\Gl}+\frac{\Gb}{r^{2}}\right)-r^{2}\Ga^{2}
+\Gg-1,\quad r=1\pm\frac{h}{2}.
\]
Let
\[
\Phi(r;\Gl,\Gb)=\ln\left(\frac{1}{1-\Gl}+\frac{\Gb}{r^{2}}\right)-
\frac{r^{2}}{r^{2}+\Gb(1-\Gl)}.
\]
Then,
\begin{equation}
  \label{tfbc1}
\begin{cases}
  \Ga^{2}\left(1+\frac{h}{2}\right)^{2}=\Phi\left(1+\frac{h}{2};\Gl,\Gb\right)+\Gg-1,\\
  \Ga^{2}\left(1-\frac{h}{2}\right)^{2}=\Phi\left(1-\frac{h}{2};\Gl,\Gb\right)+\Gg-1
\end{cases}
\end{equation}
Eliminating $\Gg$ from (\ref{tfbc1}) we obtain
\[
\Ga^{2}=\nth{2h}\left(\Phi\left(1+\frac{h}{2};\Gl,\Gb\right)
-\Phi\left(1-\frac{h}{2};\Gl,\Gb\right)\right).
\]
when $h$ is small
\[
\Ga^{2}\approx\hf\Phi'(1;\Gl,\Gb)=-\frac{\Gb(1-\Gl)(2+\Gb(1-\Gl))}{(1+\Gb(1-\Gl))^{2}}.
\]
Thus, when $(h,\Gl)\to (0,0)$, $\Gb\approx-\Ga^{2}/2$. We conclude that
$\Ga$, and, therefore, $\Gb$ must go to zero, as $\Gl\to 0$, since otherwise,
the trivial branch $\By(\Bx;h,\Gl)$, given by (\ref{tbansatz}), (\ref{psi})
will not emanate from the undeformed state. The regularity of the trivial
branch in $\Gl$ demands
that $\Ga(h,\Gl)\sim\Ga_{0}(h)\Gl$, as $\Gl\to 0$. Thus, for an arbitrary
fixed parameter $\Gb_{0}>0$ we set $\Gb=-\Gb_{0}^{2}\Gl^{2}/2$, resulting in the
explicit expression for the parameter $\Ga$:
\[
\Ga(\Gl,h)=\sqrt{\frac{\Phi(1+h/2;\Gl,-\Gb_{0}^{2}\Gl^{2}/2)
-\Phi(1-h/2;\Gl,-\Gb_{0}^{2}\Gl^{2}/2)}{2h}}.
\]
We compute
\[
\left.\dif{\Ga}{\Gl}\right|_{\Gl=0}=\frac{4\Gb_{0}}{4-h^{2}},\qquad
\left.\dif{\psi}{\Gl}\right|_{\Gl=0}=\frac{r}{2}.
\]
Therefore,
the linearized trivial branch displacement $\Bu^{h}$ is given by
\[
u^{h}_{r}=\left.\dif{y_{r}}{\Gl}\right|_{\Gl=0}=\frac{r}{2},\qquad
u^{h}_{\Gth}=\left.\dif{y_{\Gth}}{\Gl}\right|_{\Gl=0}=\frac{4\Gb_{0}rz}{4-h^{2}},\qquad
u^{h}_{z}=\left.\dif{y_{z}}{\Gl}\right|_{\Gl=0}=-z.
\]
The corresponding linear stress and its $h\to 0$ limit are
\[
\BGs_{h}=E
\begin{bmatrix}
  0 & 0 & 0\\
  0 & 0 & \frac{4\Gb_{0}r}{3(4-h^{2})}\\
  0 & \frac{4\Gb_{0}r}{3(4-h^{2})} & -1
\end{bmatrix},\qquad
\BGs^{0}=E
\begin{bmatrix}
  0 & 0 & 0\\
  0 & 0 & \frac{\Gb_{0}}{3}\\
  0 & \frac{\Gb_{0}}{3} & -1
\end{bmatrix}.
\]
These agree with formulas (\ref{lintor}), (\ref{impstr}) for $\nu=1/2$.

 \bibliography{refs}

\def\cprime{$'$} \ifx \cedla \undefined \let \cedla = \c \fi\ifx \cyr
  \undefined \let \cyr = \relax \fi\ifx \cprime \undefined \def \cprime
  {$\mathsurround=0pt '$}\fi\ifx \prime \undefined \def \prime {'}
  \fi\def\Ya{Ya}
\begin{thebibliography}{10}

\bibitem{almr63}
B.~O. Almroth.
\newblock Postbuckling behaviour of axially compressed circular cylinders.
\newblock {\em AIAA J}, 1:627--633, 1963.

\bibitem{buha66}
B.~Budiansky and J.~Hutchinson.
\newblock A survey of some buckling problems.
\newblock Technical Report CR-66071, NASA, February 1966.

\bibitem{buha79}
B.~Budiansky and J.~Hutchinson.
\newblock Buckling: Progress and challenge.
\newblock In J.~F. Besseling and van~der Heijden, editors, {\em Trends in solid
  mechanics 1979. Proceedings of the symposium dedicated to the 65 th birthday
  of W. T. Koiter}, pages 185--208. Delft University, Delft University Press
  Sijthoff \& Noordhoff International publishers, 1979.

\bibitem{bushnell81}
D.~Bushnell.
\newblock Buckling of shells-pitfall for designers.
\newblock {\em AIAA J}, 19(9):1183--1226, 1981.

\bibitem{call2000}
C.~R. Calladine.
\newblock A shell-buckling paradox resolved.
\newblock In D.~Durban, G.~Givoli, and J.~G. Simmonds, editors, {\em Advances
  in the Mechanics of Plates and Shells}, pages 119--134. Kluwer Academic
  Publishers, Dordrecht, 2000.

\bibitem{dkzoa12}
R.~Degenhardt, A.~Kling, R.~Zimmermann, F.~Odermann, and F.~de~Ara{\'u}jo.
\newblock Dealing with imperfection sensitivity of composite structures prone
  to buckling.
\newblock In S.~B. Coskun, editor, {\em Advances in Computational Stability
  Analysis}. InTech, 2012.

\bibitem{fjmm07}
I.~Fonseca, N.~Fusco, G.~Leoni, and M.~Morini.
\newblock Equilibrium configurations of epitaxially strained crystalline films:
  existence and regularity results.
\newblock {\em Arch. Ration. Mech. Anal.}, 186(3):477--537, 2007.

\bibitem{fjm02}
G.~Friesecke, R.~D. James, and S.~M{\"u}ller.
\newblock A theorem on geometric rigidity and the derivation of nonlinear plate
  theory from three-dimensional elasticity.
\newblock {\em Comm. Pure Appl. Math.}, 55(11):1461--1506, 2002.

\bibitem{fjm06}
G.~Friesecke, R.~D. James, and S.~M{\"u}ller.
\newblock A hierarchy of plate models derived from nonlinear elasticity by
  gamma-convergence.
\newblock {\em Arch. Ration. Mech. Anal.}, 180(2):183--236, 2006.

\bibitem{goev70}
D.~J. Gorman and R.~M. Evan-Iwanowski.
\newblock An analytical and experimental investigation of the effects of large
  prebuckling deformations on the buckling of clamped thin-walled circular
  cylindrical shells subjected to axial loading and internal pressure.
\newblock {\em Develop. in Theor. and Appl. Mech.}, 4:415--426, 1970.

\bibitem{grha14}
Y.~Grabovsky and D.~Harutyunyan.
\newblock Exact scaling exponents in korn and korn-type inequalities for
  cylindrical shells.
\newblock submitted.

\bibitem{grha15}
Y.~Grabovsky and D.~Harutyunyan.
\newblock Rigorous derivation of the buckling load in axially compressed
  circular cylindrical shells.
\newblock in preparation.

\bibitem{grtr07}
Y.~Grabovsky and L.~Truskinovsky.
\newblock The flip side of buckling.
\newblock {\em Cont. Mech. Thermodyn.}, 19(3-4):211--243, 2007.

\bibitem{grig10ma}
J.~Griggs.
\newblock Experimental study of buckling of thin-walled cylindrical shells.
\newblock Masters of arts thesis, Temple University, Philadelphia, PA, May
  2010.

\bibitem{hlp06}
J.~Hor{{\'a}}k, G.~J. Lord, and M.~A. Peletier.
\newblock Cylinder buckling: the mountain pass as an organizing center.
\newblock {\em SIAM J. Appl. Math.}, 66(5):1793--1824 (electronic), 2006.

\bibitem{hlc99}
G.~Hunt, G.~Lord, and A.~Champneys.
\newblock Homoclinic and heteroclinic orbits underlying the post-buckling of
  axially-compressed cylindrical shells.
\newblock {\em Computer methods in applied mechanics and engineering},
  170(3):239--251, 1999.

\bibitem{hlp03}
G.~Hunt, G.~Lord, and M.~Peletier.
\newblock Cylindrical shell buckling: a characterization of localization and
  periodicity.
\newblock {\em DISCRETE AND CONTINUOUS DYNAMICAL SYSTEMS SERIES B},
  3(4):505--518, 2003.

\bibitem{hune91}
G.~Hunt and E.~L. Neto.
\newblock Localized buckling in long axially-loaded cylindrical shells.
\newblock {\em Journal of the Mechanics and Physics of Solids}, 39(7):881 --
  894, 1991.

\bibitem{hune93}
G.~W. Hunt and E.~L. Neto.
\newblock Maxwell critical loads for axially loaded cylindrical shells.
\newblock {\em Trans. ASME}, 60:702--706, 1993.

\bibitem{lcp00}
E.~Lancaster, C.~Calladine, and S.~Palmer.
\newblock Paradoxical buckling behaviour of a thin cylindrical shell under
  axial compression.
\newblock {\em International Journal of Mechanical Sciences}, 42(5):843 -- 865,
  2000.

\bibitem{lmap2011}
M.~Lewicka, L.~Mahadevan, and M.~R. Pakzad.
\newblock The {F}\"oppl-von {K}\'arm\'an equations for plates with incompatible
  strains.
\newblock {\em Proc. R. Soc. Lond. Ser. A Math. Phys. Eng. Sci.},
  467(2126):402--426, 2011.
\newblock With supplementary data available online.

\bibitem{lmop2011}
M.~Lewicka, M.~G. Mora, and M.~R. Pakzad.
\newblock The matching property of infinitesimal isometries on elliptic
  surfaces and elasticity of thin shells.
\newblock {\em Arch. Ration. Mech. Anal.}, 200(3):1023--1050, 2011.

\bibitem{lch97}
G.~Lord, A.~Champneys, and G.~Hunt.
\newblock Computation of localized post buckling in long axially compressed
  cylindrical shells.
\newblock {\em Philosophical Transactions of the Royal Society of London.
  Series A: Mathematical, Physical and Engineering Sciences},
  355(1732):2137--2150, 1997.

\bibitem{lorenz11}
R.~Lorenz.
\newblock Die nicht achsensymmetrische knickung d{\"u}nnwandiger hohlzylinder.
\newblock {\em Physikalische Zeitschrift}, 12(7):241--260, 1911.

\bibitem{momu04}
M.~G. Mora and S.~M{\"u}ller.
\newblock A nonlinear model for inextensible rods as a low energy
  {$\Gamma$}-limit of three-dimensional nonlinear elasticity.
\newblock {\em Ann. Inst. H. Poincar\'e Anal. Non Lin\'eaire}, 21(3):271--293,
  2004.

\bibitem{tenn64}
R.~C. Tennyson.
\newblock An experimental investigation of the buckling of circular cylindrical
  shells in axial compression using the photoelastic technique.
\newblock Tech. Report 102, University of Toronto, Toronto, ON, Canada, 1964.

\bibitem{timosh14}
S.~Timoshenko.
\newblock Towards the question of deformation and stability of cylindrical
  shell.
\newblock {\em Vesti Obshestva Tekhnologii}, 21:785--792, 1914.

\bibitem{tiwk59}
S.~Timoshenko and S.~Woinowsky-Krieger.
\newblock {\em Theory of plates and shells}, volume~2.
\newblock McGraw-hill New York, second edition, 1959.

\bibitem{wms65}
V.~I. Weingarten, E.~J. Morgan, and P.~Seide.
\newblock Elastic stability of thin-walled cylindrical and conical shells under
  axial compression.
\newblock {\em AIAA J}, 3:500--505, 1965.

\bibitem{yama84}
N.~Yamaki.
\newblock {\em Elastic Stability of Circular Cylindrical Shells}, volume~27 of
  {\em Appl. Math. Mech.}
\newblock North Holland, 1984.

\bibitem{yosh55}
Y.~Youshimura.
\newblock On the mechanism of buckling of a circular shell under axial
  compression.
\newblock Technical Report 1390, National advisory committee for aeronautics,
  Washington, DC, 1955.

\bibitem{zmc02}
E.~Zhu, P.~Mandal, and C.~Calladine.
\newblock Buckling of thin cylindrical shells: an attempt to resolve a paradox.
\newblock {\em International Journal of Mechanical Sciences}, 44(8):1583 --
  1601, 2002.

\end{thebibliography}

\end{document}